\documentclass[a4paper,reqno,11pt]{amsart}

\usepackage{amsmath,amsfonts,amssymb}
\usepackage{graphics,color}
\usepackage[latin1]{inputenc}
\usepackage[colorlinks = true,
            linkcolor = blue,
            urlcolor  = blue,
            citecolor = blue,
            anchorcolor = blue]{hyperref}

\newtheorem{theorem}{Theorem}[section]
\newtheorem{lemma}[theorem]{Lemma}

\makeatletter
\@namedef{subjclassname@2020}{\textup{2020} Mathematics Subject Classification}
\makeatother

\numberwithin{equation}{section}
\newtheorem{proposition}[theorem]{Proposition}
\newtheorem{corollary}[theorem]{Corollary}
\newtheorem{definition}[theorem]{Definition}

\newtheorem{remark}[theorem]{Remark}
\usepackage{cite}
\numberwithin{equation}{section}

\allowdisplaybreaks[4]
\begin{document}
	
\title[Nonlinear Schr\"{o}dinger equations with mixed fractional Laplacians]
{Normalized solutions for the NLS equation with mixed fractional Laplacians and combined nonlinearities}

\author{Shubin Yu}
\address{School of Mathematics and Statistics, Southwest University, Chongqing 400715, People's Republic of China.}
\email{yshubin168@163.com}

\author{Chen Yang}
\address{School of Mathematics and Statistics, Southwest University, Chongqing 400715, People's Republic of China.}
\email{yangchen6858@163.com}

\author{Chun-Lei Tang$^*$}
\address{School of Mathematics and Statistics, Southwest University, Chongqing 400715, People's Republic of China.}
\email{tangcl@swu.edu.cn}

\footnotetext[1]{Corresponding author.}
\footnotetext[2]{Project supported by the National Natural Science Foundation of China (No.12371120).}


\begin{abstract}
We look for normalized solutions to the nonlinear Schr\"{o}dinger equation with mixed  fractional Laplacians and combined nonlinearities
$$
\left\{\begin{array}{ll}
(-\Delta)^{s_{1}} u+(-\Delta)^{s_{2}} u=\lambda u+\mu |u|^{q-2}u+|u|^{p-2}u \ \text{in}\;{\mathbb{R}^{N}}, \\[0.1cm]
\int_{\mathbb{R}^{N}}|u|^2\mathrm dx=a^2,
\end{array}
\right.
$$
where $N\geq 2,\;0<s_2<s_1<1, \mu>0$ and $\lambda\in\mathbb R$ appears as an unknown Lagrange multiplier.
We mainly focus on some special cases, including fractional Sobolev subcritical or critical exponent.
More precisely, for $2<q<2+\frac{4s_2}{N}<2+\frac{4s_1}{N}
      <p<2_{s_1}^{\ast}:=\frac{2N}{N-2s_1}$, we prove that the
above problem has at least two solutions: a ground state with negative energy and a solution of mountain pass type with positive energy. For $2<q<2+\frac{4s_2}{N}$ and $p=2_{s_1}^{\ast}$, we also obtain  the existence of ground states.
Our results extend some previous ones of Chergui et al. (Calc. Var. Partial Differ. Equ., 2023) and Luo et al. (Adv. Nonlinear Stud., 2022).
\end{abstract}
\maketitle
\vspace {-1cm}

\section{Introduction}
In this paper, we consider the following nonlinear Schr\"{o}dinger equation with mixed fractional  Laplacians and combined nonlinearities
\begin{equation}\label{eqn:Mixed-fractioanl-CN}
\left\{\begin{array}{ll}
(-\Delta)^{s_{1}} u+(-\Delta)^{s_{2}} u=\lambda u+\mu |u|^{q-2}u+|u|^{p-2}u \ \text{in}\ {\mathbb{R}^{N}}, \\[0.1cm]
\int_{\mathbb{R}^{N}}|u|^2\mathrm dx=a^2,
\end{array}
\right.
\end{equation}
where $N\geq 2,0<s_2<s_1<1$, $2<q<2+\frac{4s_2}{N}<2+\frac{4s_1}{N}
<p\leq2_{s_1}^{\ast}:=\frac{2N}{N-2s_1}$, $\mu>0$ and $\lambda\in\mathbb R$ is an unknown Lagrange multiplier.
The fractional Laplacian $(-\Delta)^s$ with $0<s<1$
can be viewed as a pseudo-differential
operator of symbol $|\xi|^{2s}$, that is, $ (-\Delta)^su=\mathcal{F}^{-1}(|\xi|^{2s}\mathcal{F}(u))$\;for $\xi\in \mathbb{R}^{N}$,\;where $\mathcal{F}$\;denotes the Fourier transform, see \cite{NPV2012} for more details.

Note that if
problem \eqref{eqn:Mixed-fractioanl-CN} only involves single fractional laplacian, it will be reduced to the well-known fractional Schr\"{o}dinger equation with prescribed mass and the existence of normalized solutions has been widely studied, see for instance \cite{Luo2020,Appolloni2021,Zhang2022,Zhen2022} and the references therein. In recent years,
 based on their application in biology, such as describing the diffusion in an ecological niche subject to nonlocal dispersals,
the mixed
 fractional
Laplacians have received more attention.
 Indeed,
 the mixed fractional laplacians are the outcome of the superposition of two long-range L\'{e}vy processes or a classical Brownian motion and a long-range process.
 The population diffuses according to two or more types of nonlocal dispersals, modeled by L\'{e}vy flights and encoded by two or more fractional Laplacians with two different powers.
 We refer the reader to recent papers \cite{Valdinoci-arXiv,Chen-arXiv} for this topic.

Regarding the existence of normalized solutions for problem \eqref{eqn:Mixed-fractioanl-CN}, the current works
are presented in \cite{Chergui-Gou-Hajaiej-2023-CVPDE,luo-Hajaiej-NA-2022}, where the authors only considered
\eqref{eqn:Mixed-fractioanl-CN} with $\mu=0$, that is,
\begin{equation}\label{eqn:Mixed-fractioanl-single-power}
\left\{\begin{array}{ll}
(-\Delta)^{s_{1}} u+(-\Delta)^{s_{2}} u=\lambda u+|u|^{p-2}u \ \text{in}\ {\mathbb{R}^{N}}, \\[0.1cm]
\int_{\mathbb{R}^{N}}|u|^2\mathrm dx=a^2,
\end{array}
\right.
\end{equation}
where $N\geq 1,0<s_2<s_1<1$ and $2<p<2_{s_1}^{\ast}$. In view of the well-known Gagliardo-Nirenberg inequality (see Lemma \ref{lem:GN-inequality}),
the number $\bar p:=2+\frac{4s_1}{N}$ is the $L^2$-critical exponent. If $2<p<\bar p$, the corresponding constrained functional is bounded from below and
the authors in \cite{luo-Hajaiej-NA-2022}
gave the
sharp existence of
 ground states in the sense of Definition \ref{def:nonexistence-ground-state} below.
If $\bar p<p<2_{s_1}^*$, on the contrary, the constrained functional is unbounded from below. At this point, by introducing the Pohozaev manifold, Chergui et al. \cite{Chergui-Gou-Hajaiej-2023-CVPDE} established the existence of ground states at least the $L^2$-critical regime (i.e. $\bar p\leq p<2_{s_1}^*$).
 Moreover, they also obtained the multiplicity
of bound state solutions and
the orbital instability of ground
states. In particular, due to the
appearance of low order fractional Laplacian $(-\Delta)^{s_{2}}$,
 it is necessary to consider the impact of exponents $\tilde p:=2+\frac{4s_2}{N}$ and $2_{s_2}^*:=\frac{2N}{N-2s_2}$, which
 makes
the problem \eqref{eqn:Mixed-fractioanl-single-power} and also problem  \eqref{eqn:Mixed-fractioanl-CN} different from the nonlinear Schr\"{o}dinger equations with single fractional Laplacian.

Compared to the above results, we are concerned with  problem
\eqref{eqn:Mixed-fractioanl-CN} with $\mu>0$.
Obviously, the energy functional of \eqref{eqn:Mixed-fractioanl-CN} can be defined by
$$
E_{\mu}(u):=\frac{1}{2}\int_{\mathbb{R}^{N}}|(-\Delta)^{\frac{s_{1}}{2}} u|^2 dx+\frac{1}{2}\int_{\mathbb{R}^{N}}|(-\Delta)^{\frac{s_{2}}{2}} u|^2 dx-
\frac{1}{p}\int_{\mathbb{R}^{N}}|u|^{p}dx-\frac{\mu}{q}
\int_{\mathbb{R}^{N}}|u|^{q}dx$$
and the solutions of problem \eqref{eqn:Mixed-fractioanl-CN}
can be obtained as critical points of the energy functional $E_\mu$
restricted on the constraint
$$
S_a=\{u\in H^{s_{1}}(\mathbb{R}^2):|u|_{2}^{2}=a^2>0\}.
$$
It is obvious that $E_{\mu}\in C^{1}(H^{s_{1}}(\mathbb{R}^N),\mathbb{R})$.
If $\mu>0$, the nonlinearities are seen as combined cases, which were introduced by Soave \cite{SN2020,Soave-2020-JFA}.
Indeed, as stated in \cite{SN2020}, the interplay between $L^2$-subcritical, $L^2$-critical and $L^2$-supercritical nonlinearities has a
deep impact on the geometry of the functional $E_\mu|_{S_a}$ and the existence normalized solutions.
In this context, there may be two critical points of $E_\mu|_{S_a}$
if $q$ and $p$ are between the two sides of the $L^2$-critical exponent $\bar p=2+\frac{4s_1}{N}$. One of the critical points will be the ground state of problem \eqref{eqn:Mixed-fractioanl-CN} in the following sense.
\begin{definition}\label{def:nonexistence-ground-state}\rm
We say that $u_a$ is a ground state to \eqref{eqn:Mixed-fractioanl-CN} on $S_a$ if it is a solution to \eqref{eqn:Mixed-fractioanl-CN}
having minimal energy
among all the solutions which belongs to $S_a$, that is,
$$
d E_{\mu}|_{S_a}(u_a)=0\ \mbox{and}\
E_{\mu}(u_{a})=\inf\Bigl\{E(u):u\in S_a,\ dE_{\mu}|_{S_a}(u)=0\Bigr\}.
$$
\end{definition}

In light of the above discussion,
we mainly focus on the following three cases:
\begin{enumerate}
  \item $2<q<2+\frac{4s_2}{N}<2+\frac{4s_1}{N}<p< 2_{s_1}^{\ast}$;
  \item $2<q<2+\frac{4s_2}{N}<p= 2_{s_1}^{\ast}$;
  \item $2+\frac{4s_1}{N}<q<p= 2_{s_1}^{\ast}$.
\end{enumerate}
Note that we restrict $2<q<2+\frac{4s_2}{N}$ in cases $(1)$ and $(2)$. Indeed, if $2+\frac{4s_2}{N}<q<2+\frac{4s_1}{N}$, the geometry of the functional $E_\mu|_{S_a}$ is very complex, and we will not discuss it in this article.
We refer the reader to \cite{Chen-MA-2024} for some inspirations.
%For $q=2+\frac{4s_2}{N}$, our results can be applied with additional assumptions.  For the sake of simplicity, we will also not provide details.
Based on the above three cases, we are interested in the following three questions:
\begin{enumerate}
  \item [$(Q_1)$] Does $E_\mu|_{S_a}$ have two critical points, one is the local minimizer and the other is the mountain pass type if $p$, $q$ satisfy case $(1)$?
  \item [$(Q_2)$]
If $(Q_1)$ is solvable, can the results be obtained for case (2)?
  \item [$(Q_3)$] Does $E_\mu|_{S_a}$ have a critical point of mountain pass type if $p$, $q$ satisfy case $(3)$?
\end{enumerate}
Moreover, in present paper, we shall provide an affirmative answer
  for $(Q_1)$ and partially resolve $(Q_2)$.
For $(Q_3) $, which appears as an open question, we will provide some comments at the end of this section.

First, we
give
the main result for $p$, $q$ satisfy case $(1)$.
To this end, we assume that $a, \mu>0$ satisfy the following the conditions:
 \begin{itemize}
   \item [$(A_0)$]:
$$
a^{\frac{2(p-2)s_1}{N(p-2)-4s_1}}
<\left(\frac{s_1-s_2}{C_{N,s_1,p}}\right)
^{\frac{2s_1}{N(p-2)-4s_1}}
\left(\frac{s_1-\kappa_p}{\kappa_p-s_2}
\right)^{\frac{s_1}{2(s_1-s_2)}}\ \mbox{if}\ p> \frac{2N}{N-2s_2};
$$
\item [$(A_1)$]:$$\begin{aligned}
   &\mu a^{q-\frac{N(q-2)}{2 s_1}+ \left(p-\frac{N(p-2)}{2 s_1}\right)\cdot{\frac{4s_1-N(q-2)}{N(p-2)-4s_1}}}\\
   &\ \ <\bigg(\frac{p(4s_1-N(q-2))}{2NC_{N,s_1,p}(p-q)}\bigg)
   ^{\frac{4s_1-N(q-2)}{N(p-2)-4s_1}}
   \bigg(\frac{N(p-2)-4s_1}{2N(p-q)}\bigg)\frac{q}{C_{N,s_1,q}};
 \end{aligned}
$$
   \item [$(A_2)$]:
\begin{equation*}
\begin{aligned}
&a^{\frac{2qs_1 -N(q-2)}{4s_1 -N(q-2)}+{\frac{2p s_1 -N(p-2)}{N(p-2)-4s_1}}}\mu^{\frac{2 s_1}{4s_1 -N(q-2)}}\\
&\ \ <\bigg(\frac{s_1(p\kappa_p -2s_1)}{ C_{N,s_1,q} \kappa_q(p\kappa_p -q\kappa_q )}\bigg)^{\frac{2 s_1}{4s_1 -N(q-2)}} \bigg(\frac{s_1(2s_1 -q\kappa_q)}{C_{N,s_1,p}\kappa_p(p\kappa_p - q\kappa_q )}\bigg)^{\frac{2 s_1}{N(p-2)-4s_1}}.
\end{aligned}
\end{equation*}
  \end{itemize}
As we will see, this result states that $(Q_1)$ is solved and also implies that problem \eqref{eqn:Mixed-fractioanl-CN}
has two solutions: a ground state with negative energy and a solution of mountain pass type with positive energy.
\begin{theorem}\label{Thm:q-p-mass-subcritical-supcritical}
Let $N\geq2,\,2<q<2+\frac{4s_2}{N}<2+\frac{4s_1}{N}<p<2_{s_1}^{\ast} $, $a,\mu>0$  satisfy conditions $(A_1)-(A_2)$ and $(A_0)$ holds if $p>\frac{2N}{N-2s_2}$.
Then
\begin{itemize}
 \item [($i$)]
 $E_{\mu}|_{S_a}$ has a critical point $\tilde{u}$ at negative level $\gamma_{\mu}(a)<0$, which is an interior local minimizer of $E_{\mu}$ on the set
$$A_{R_0}:=\left\{u\in S_a:|(-\Delta)^{\frac{s_1}{2}} u|_{2}<R_0\right\}$$
  for a suitable $R_0>0$.
 \item [($ii$)] $E_{\mu}|_{S_a}$ has a second critical point of mountain pass type $\hat{u}$ at positive level $\sigma_{\mu}(a)>0$.
 \item [($iii$)] Both $\tilde{u}$ and $\hat{u}$ are nonnegative, radially symmetric, and solve
\eqref{eqn:Mixed-fractioanl-CN} for suitable $\tilde{\lambda}$, $\hat{\lambda}<0$. In particular, $\tilde{u}$  is a ground state of \eqref{eqn:Mixed-fractioanl-CN}  and radially decreasing.
\end{itemize}
\end{theorem}

% \begin{itemize}
%   \item [$(\bar A_0)$]:
%$$
%a\mu^{\frac{2Ns_1}{N(2qs_1 -N(q-2))}}<\left(\frac{1}{2K_0}\right)^{\frac{N}{4s_1 }}
%$$
%with
%$$
%\begin{aligned}
%K_0:&=
%\frac{1}{2_{s_1}^{*}}\frac{1}{{{\mathcal{S} }^{\frac{2_{s_1}^{*}}{2}}}}\left(\frac{C_{N,s_1,q}(4s_1 -N(q-2))}{2s_1q}\cdot\frac{N \mathcal{S}^{\frac{2_{s_1}^{*}}{2}}}{2s_1}\right)^{\frac{8 s_1 ^{2}}{N(2qs_1 -N(q-2))}}\\
%&\ \ +\frac{C_{N,s_1,q}}{q}\left(\frac{C_{N,s_1,q}(4s_1 -N(q-2))}{2s_1q}\cdot\frac{N \mathcal{S}^{\frac{2_{s_1}^{*}}{2}}}{2s_1}\right)^{\frac{(N-2s_1 )(N(q-2)-4s_1 )}{N(2qs_1 -N(q-2))}}
%\end{aligned}
%$$
%\end{itemize}

In the following, we consider the case $(2)$, i.e., $2<q<2+\frac{4s_2}{N}<p= 2_{s_1}^{\ast}$.
Since $p=2_{s_1}^{*}$ is the
Sobolev critical exponent, we introduce the ideas in \cite{Jeanjean2022} to obtain a  local minimizer for $E_\mu|_{S_a}$, which is more direct compared to the  previous case.
However, it is difficult to determine the existence of mountain type solutions using existing methods, which is the essential challenge brought by Sobolev critical exponent. Now we give the corresponding result, which means that question $(Q_2)$ is partially solved and further implies that \eqref{eqn:Mixed-fractioanl-CN} has a ground state.

\begin{theorem}\label{Thm:q-mass-subcritical-p-sobolev-critical}
Let $N\geq2,\,2<q<2+\frac{4s_2}{N}< p=2_{s_1}^{*}$. For any $\mu>0$ there exists $a_0=a_0 (\mu)>0$ such that for any $a\in (0,a_0)$,
 $E_{\mu}|_{S_a}$ has a critical point $u_a$ at negative level $\overline{\gamma_{\mu}}(a)<0$, which is an interior local minimizer of $E_{\mu}$ on the set
 $$A_{a,\rho_0}:=\left\{u\in S_a:|(-\Delta)^{\frac{s_1}{2}} u|_{2}<\rho_0\right\}$$ for a suitable $\rho_0>0$.
Furthermore, there exists $\bar a_0=\bar a_0(\mu)>0$ such that if $a<\min\{a_0,\bar a_0\}$, then $u_a$ is a ground state of \eqref{eqn:Mixed-fractioanl-CN} for suitable $\bar \lambda\in\mathbb R$, and $u_a$ is nonnegative and  radially decreasing.
\end{theorem}
\begin{remark}\rm We remark that $a_0=a_0(\mu)$ and $\bar a_0=\bar a_0(\mu)$ are explicit for $\mu>0$, see \eqref{eqn:definition-a-0-1} and \eqref{eqn:a-0-condition-2}.
Precisely, the condition
$a<\min\{a_0,\bar a_0\}$  can be written in the form
$$
a\mu^{\frac{2 s_1}{2qs_1 -N(q-2)}}<\min\{C_0,C_1\},
$$
where $C_0>0$ and $C_1>0$ are explicit constants independent of $a$ and $\mu$. This implies that $a$ and $\mu$ are not necessarily small, since we can take one between $a$ and $\mu$
as large as we want, provided that the other is sufficient small.
\end{remark}
\begin{remark}\rm
On the other hand, observe that conditions $(A_1)$ and $(A_2)$ imply that $a$ and $\mu$
are also not necessarily small in Theorem \ref{Thm:q-p-mass-subcritical-supcritical} if $p\leq \frac{2N}{N-2s_2}$. However, if $\frac{2N}{N-2s_2}<p<2_{s_1}^*$, $(A_0)$ yields that $a>0$ must be small.
This condition  appears when we  determine the Lagrange multiplier $\lambda_n\rightarrow \lambda<0$ as $n\rightarrow\infty$ in establishing the compactness of special Palais-Smale sequences, see Lemma \ref{lem:mass-supercritical-setting-PS-compact}.
It is not difficult to see that, for $p<2_{s_1}^*$,
 we can also obtain a local minimizer that corresponds to the ground state of \eqref{eqn:Mixed-fractioanl-CN}  by repeating the procedure in Section \ref{sec:sobolev-critical}.
As a result, the condition $(A_0)$  for Theorem \ref{Thm:q-p-mass-subcritical-supcritical}-$(i)$ can be removed,
 but it cannot be removed for Theorem \ref{Thm:q-p-mass-subcritical-supcritical}-$(iii)$  since the sign of the multiplier $\tilde \lambda$ is strictly dependent on $a>0$ small and $p<2_{s_1}^*$. This further indicates that the sign of $\bar \lambda$ in Theorem \ref{Thm:q-mass-subcritical-p-sobolev-critical} cannot be determined due to $p=2_{s_1}^*$.
\end{remark}
\begin{remark}\rm
Finally, we provide some comments for open question $(Q_3)$.
After careful analysis, we find that the only difficulty is to determine the compactness of special Palais-Smale sequences.
 Indeed, the sequences can be obtained by the arguments in \cite{JEAN-NA-1997} and one can confirm that the mountain pass level is less than  the threshold  $\frac{s_1}{N}\mathcal S^{\frac{N}{2s_1}}$ by making $\mu>0$ sufficiently large.
Thus,
in view of above remark, if the sign of corresponding Lagrange multiplier can be ascertained, then question $(Q_3)$ will be solved.
\end{remark}

The remainder of this paper is organized as follows. In Section \ref{sec:pre}, we give  some  preliminaries.
 Section \ref{sec:sobolev-subcritical}
  is devoted to accomplishing the proof of Theorem \ref{Thm:q-p-mass-subcritical-supcritical} and obtaining the existence of two solutions of problem \eqref{eqn:Mixed-fractioanl-CN}. In Section \ref{sec:sobolev-critical}, we focus on the existence of an
interior local minimizer and complete the proof of Theorem \ref{Thm:q-mass-subcritical-p-sobolev-critical}.

  Throughout the paper,  we
make use of the following notations:
\begin{itemize}
  \item The fractional Sobolev space $H^s(\mathbb{R}^N)$ is defined for $0<s<1$ by  $H^s(\mathbb{R}^N):=\{u\in L^2(\mathbb{R}^N):(-\Delta)^{\frac{s}{2}}u\in L^2(\mathbb{R}^N)\},$
 which is a Hilbert space endowed with norm
 $$
 \quad\|u\|^2_{ H^s(\mathbb{R}^N)}:=|(-\Delta)^{\frac{s}{2}}
 u|_2^2+|u|_2^2=\int_{\mathbb{R}^{N}}
\int_{\mathbb{R}^{N}}\frac{|u(x)-u(y)|^{2}}{|x-y|^{N+2s}}dxdy
+|u|_2^2.
$$
  \item $H_r^{s}(\mathbb R^N)$ is denoted by the radial subspace of $ H^s(\mathbb{R}^N)$ and $S_{a,r}:=S_a\cap H_r^{s}(\mathbb R^N).$
  \item The usual norm in the Lebesgue space $L^r(\mathbb R^N)$ is denoted by $|u|_r$ with $2\leq r\leq 2_{s_1}^*$.
  \item For any $x\in \mathbb R^N$ and $R>0$, $B(x,R):= \{y\in\mathbb R^N: |x-y|<R\}.$
\end{itemize}

\section{Preliminaries}\label{sec:pre}
 In this  section, we give the workspace and  some preliminaries.
For $0<s_2<s_1<1,\;N\geq2$, we recall from \cite[Remark 1.4.1]{Cazenave-2003} that $H^{s_1}(\mathbb{R}^N)\hookrightarrow H^{s_2}(\mathbb{R}^N)$. As a result, we know that
$$
H^{s_1}(\mathbb{R}^N)\cap H^{s_2}(\mathbb{R}^N)=H^{s_1}(\mathbb{R}^N).
$$
Therefore, in present paper, we will work in the Sobolev space $H^{s_1}(\mathbb{R}^N)$.
Furthermore, for $N\geq 2$, it is well-know that
$$
H^{s_1}(\mathbb{R}^N)\hookrightarrow L^q(\mathbb{R}^N)\ \text{continuously},\ \forall q\in\left[2,\frac{2N}{N-2s_1}\right]
$$
and
$$
H_r^{s_1}(\mathbb{R}^N)\hookrightarrow L^q(\mathbb{R}^N)\ \text{compactly},\ \forall q\in\left(2,\frac{2N}{N-2s_1}\right).
$$
Moreover, we recall the fractional Gagliardo-Nirenberg inequality and Sobolev inequality, which will be frequently used in subsequent arguments. We refer to
 \cite{Boulenger-Himmelsbach-2016-JFA} and \cite{Cotsiolis2004}, respectively.

 \begin{lemma}\label{lem:GN-inequality}
Let $0<s<1$, $N>2s_1$ and $2\leq p\leq\frac{2N}{N-2s_1}$. Then for any $u\in H^{s_1}(\mathbb R^N)$,
\begin{equation}\label{eqn:G-N-equality}
\begin{aligned}
\int_{\mathbb{R}^N}|u|^pdx&\leq C_{N,s_1,p}\left(\int_{\mathbb{R}^N}|(-\Delta)^{\frac {s_1}{2}}u|^2dx\right)^{\frac{N(p-2)}{4s_1}}
\left(\int_{\mathbb{R}^N}|u|^2dx\right)^{\frac p2-\frac{N(p-2)}{4s_1}},
\end{aligned}
\end{equation}
 where $C_{N,s_1,p}>0$ denotes the optimal constant. In particular,
 there exists an optimal constant $\mathcal S>0$ depending only on $N$ such that
 \begin{equation}\label{eqn:sobolev-critical-inequality}
\mathcal{S} |u|_{2_{s_1}^{*}}^2\leq |(-\Delta)^{\frac{s_1}{2}}u|_2 ^2,\ \forall u\in H^{s_1}(\mathbb R^N).
 \end{equation}
\end{lemma}
\noindent
Here we remark that
$C_{N,s_1,p}=\mathcal S^{-\frac{2}{2_{s_1}^*}} $ in the sense of inequalities \eqref{eqn:G-N-equality} and \eqref{eqn:sobolev-critical-inequality}.

In the following, according to the ideas in \cite{SN2020}, for $u\in S_a$ and $t\in \mathbb{R}$, we define the dilation
\begin{equation}\label{eqn:u-e-t}
(t \star u)(x):=e^{\frac N2t}u(e^tx),
\end{equation}
which preserves the $L^2$-norm, i.e.,
$t\star u\in S_a$. Hence it is natural to study the fiber maps
\begin{equation}\label{E-fiber-maps}
\begin{aligned}
\Phi_u^\mu(t)&:=E_\mu(t \star u)\\
&=\frac{e^{2s_{1}t}}2\int_{\mathbb{R}^N}
|(-\Delta)^{\frac{s_1}{2}}u|^2dx+\frac{ e^{2s_{2}t}}2\int_{\mathbb{R}^N}|(-\Delta)^{\frac{s_2}{2}}u|^2dx\\
&\ \ -\frac{e^{p\kappa_pt}}p\int_{\mathbb{R}^N}|u|^pdx-
\mu\frac{e^{q\kappa_qt}}q\int_{\mathbb{R}^N}|u|^qdx,
\end{aligned}
\end{equation}
where
$$
\kappa_p:=\frac{N(p-2)}{2p}\ \ \mbox{for}\ \ p\in(2,2_{s_1}^*].
$$
Now we introduce the Pohozaev
manifold
\begin{equation}\label{eqn:P-Pohozaev-set}
\mathcal{P}_{a,\mu}=\begin{Bmatrix}u\in S_a:P_\mu(u)=0\end{Bmatrix},
\end{equation}
where $P_{\mu}(u)=0$ is the Pohozaev identity of problem \eqref{eqn:Mixed-fractioanl-CN} and
\begin{equation}\label{eqn:P-Pohozaev-expression}
P_{\mu}(u)=s_1\int_{{\mathbb{R}^{N}}}|(-\Delta)^{\frac{s_1}{2}} u|^{2}dx+ s_2\int_{{\mathbb{R}^{N}}}|(-\Delta)^{\frac{s_2}{2}} u|^{2}dx-\kappa_{p}\int_{{\mathbb{R}^{N}}}|u|^{p}dx-\mu\kappa_{q}
\int_{{\mathbb{R}^{N}}}|u|^{q}dx,
\end{equation}
see for example \cite{Chergui-Gou-Hajaiej-2023-CVPDE}.
It is well-known that any critical point of $E_\mu|_{S_a}$ stays in $\mathcal{P}_{a,\mu}.$ Moreover, we have
\begin{proposition}\label{pro:t-p-uniquess}
Let $u\in S_a$, then $t\in\mathbb{R}$ is a critical point for $\Phi_{u}^{\mu}$ if and only if $t\star u\in \mathcal{P}_{a,\mu}$. Moreover, the map
\begin{equation}\label{eqn:s-u-continous-Hs1}
(t,u)\in\mathbb{R}\times H^{s_1}(\mathbb{R}^{N})\mapsto(t\star u)\in H^{s_1}(\mathbb{R}^{N})\ \text{is continuous.}
\end{equation}
\end{proposition}
\begin{proof}
The proof directly follows
the definition of $\mathcal{P}_{a,\mu}$ and \cite[Lemma 3.5]{Bartsch-Soave-CV-2019}. So we omit the details.
\end{proof}
Obviously, the above proposition implies that $u\in \mathcal{P}_{a,\mu}$ if and only if $0$ is a critical point of $\Phi_{u}^{\mu}$.
In this direction, we consider the decomposition of $\mathcal{P}_{a,\mu}$ into the disjoint union $\mathcal{P}_{a,\mu}=\mathcal{P}_{+}\cup\mathcal{P}_{0}
\cup\mathcal{P}_{-}$, where
\begin{equation*}
\begin{aligned}
\mathcal{P}_{+}&:=\left\{u\in\mathcal{P}_{a,\mu}:2s_{1}^2
|(-\Delta)^{\frac{s_1}{2}}u|_{2}^{2}+2 s_{2}^2|(-\Delta)^{\frac{s_2}{2}}u|_{2}^{2}>\mu q\kappa_{q}^{2}|u|_{q}^{q}+p\kappa_{p}^{2}|u|_{p}^{p}\right\}\\
&\;=\left\{u\in\mathcal{P}_{a,\mu}:(\Phi_{u}^{\mu})^{\prime\prime}(0)>0\right\};\\
\mathcal{P}_{-}&:=\left\{u\in\mathcal{P}_{a,\mu}:2s_{1}^2|
(-\Delta)^{\frac{s_1}{2}}u|_{2}^{2}+2 s_{2}^2|(-\Delta)^{\frac{s_2}{2}}u|_{2}^{2}<\mu q\kappa_{q}^{2}|u|_{q}^{q}+p\kappa_{p}^{2}|u|_{p}^{p}\right\}\\
&\;=\left\{u\in\mathcal{P}_{a,\mu}:(\Phi_{u}^{\mu})
^{\prime\prime}(0)<0\right\};\\
\mathcal{P}_{0}&:=\left\{u\in\mathcal{P}_{a,\mu}:2s_{1}^2
|(-\Delta)^{\frac{s_1}{2}}u|_{2}^{2}+2 s_{2}^2|(-\Delta)^{\frac{s_2}{2}}u|_{2}^{2}=\mu q\kappa_{q}^{2}|u|_{q}^{q}+p\kappa_{p}^{2}|u|_{p}^{p}\right\}\\
&\;=\left\{u\in\mathcal{P}_{a,\mu}:(\Phi_{u}^{\mu})^{\prime\prime}(0)=0\right\}.
\end{aligned}
\end{equation*}
As we will see, this decomposition will play a crucial role in this article.

\section{$L^2$-supercritical leading term with a subcritical perturbation}\label{sec:sobolev-subcritical}
 In this section, for $N\geq2$ and $0<s_2<s_1<1$, we focus on the case of $2<q< 2+\frac{4s_2}{N}<\bar{p}<p<2_{s_1}^*$ and
 prove that problem \eqref{eqn:Mixed-fractioanl-CN}  possesses two solutions, that is, Theorem \ref{Thm:q-p-mass-subcritical-supcritical} is obtained.
From now on, we omit the dependence of $E_\mu,\;P_\mu,\; \mathcal{P}_{a,\mu}$ and $\Phi_{u}^{\mu}$ on these
quantities, writing simply $E,\;P,\;\mathcal{P}$ and $\Phi_u$.

\subsection{Compactness of Palais-Smale sequences}
\ \\
In this subsection, we prove a prior compactness result for special Palais-Smale sequences (in the sense that $\{u_n\}$ is Palais-Smale sequences for $E|_{S_a}$ and satisfies $P(u_n)\rightarrow0$ as $n\rightarrow\infty$).

\begin{lemma}\label{lem:mass-supercritical-setting-PS-compact}
Let $N\geq 2$, $\mu>0$ and $2<q<\bar p<p<2_{s_1}^{\ast}$. Let $\{u_n\}\subset S_{a,r}$ be a Palais-Smale sequence for $E|_{S_a}$ at level $c\neq 0$, and suppose in addition that:
\begin{itemize}
  \item [$(i)$] $P(u_n)\to0\mathrm{~as~}n\to\infty;$
  \item [$(ii)$]  $(A_0)$ holds if $p> \frac{2N}{N-2s_2}$.
\end{itemize}
Then up to a subsequence $u_{n}\to u$ strongly in $H^{s_1}(\mathbb{R}^{N})$ and $u\in S_a$ is a  radial solution to \eqref{eqn:Mixed-fractioanl-CN} for some $\lambda<0$.
\end{lemma}
\begin{proof}
The proof is divided into four steps.

\textbf{\emph{Step} 1. Boundedness of $\{u_n\}$ in $H^{s_1}(\mathbb{R}^N)$}.
Note that $P(u_{n})\to 0$, then we observe that
 \begin{equation*}
|u_n|_{p}^{p}=\frac{s_{1}}{\kappa_{p}}|(-\Delta)^{\frac{s_1}{2}} u_n|_{2}^{2}+\frac{s_{2}}{\kappa_{p}}|(-\Delta)^{\frac{s_2}{2}}u_n|_{2}^{2}
-\frac{\mu \kappa_{q}}{\kappa_{p}}|u_n|_{q}^{q}+o_{n}(1).
\end{equation*}
Thus, we have
\begin{equation*}
E(u_n)=\bigg(\frac{1}{2}-\frac{s_{1}}{\kappa_{p}p}\bigg)|(-\Delta)^{\frac{s_1}{2}} u_n|_{2}^{2}+\bigg(\frac{1}{2}-\frac{s_{2}}{\kappa_{p}p}\bigg)|(-\Delta)^{\frac{s_2}{2}}u_n|_{2}^{2}
-\frac{\mu}{q}\bigg(1-\frac{\kappa_{q} q}{\kappa_{p} p}\bigg)|u_n|_{q}^{q}+o_{n}(1).
\end{equation*}
Since $2<q<\bar p<p<2_{s_1}^{\ast}$, we know
all the coefficients inside the brackets are positive. Then, by \eqref{eqn:G-N-equality} with $s=s_1$, we infer that
\begin{equation*}
c+1\geq E(u_n)\geq \bigg(\frac{1}{2}-\frac{s_{1}}{\kappa_{p}p}\bigg)|(-\Delta)^{\frac{s_1}{2}} u_{n}|_{2}^{2}-
\frac{\mu}{q}\bigg(1-\frac{\kappa_{q}q}{\kappa_{p} p}\bigg)a^{q-\frac{N(q-2)}{2s_{1}}}C_{N,s_{1},q}
|(-\Delta)^{\frac{s_{1}}{2}}u_n|_{2}^{\frac{N(q-2)}{2s_1}},
\end{equation*}
which implies that $\{u_n\}$ is bounded in $H^{s_1}(\mathbb{R}^N)$ due to $\frac{N(q-2)}{2s_1}<2$.

\textbf{\emph{Step} 2. There exists $\lambda\in\mathbb{R}$ such that $\lambda_n\to \lambda$}.
Since the embedding $H_{r}^{s_1}(\mathbb{R}^N)\hookrightarrow L^s(\mathbb{R}^N)$  is compact for $s\in(2,2_{s_{1}}^{\ast})$, then we
deduce that there exists $u\in H_{r}^{s_1}(\mathbb{R}^N)$ such that, up to a subsequence, $u_{n}\rightharpoonup u$ weakly in $H^{s_1}(\mathbb{R}^{N})$,
$u_{n}\to u$ strongly in $L^s(\mathbb{R}^N)$ for $s\in(2,2_{s_{1}}^{\ast})$ and a.e. in $\mathbb{R}^N$. By Step 1, we know that $\{u_{n}\}$ is a bounded Palais-Smale
sequence of $E|_{S_a}$, then by the Lagrange multipliers rule there exists $\lambda_n \in \mathbb{R}$ such that
\begin{equation}\label{eqn:un-Lagrange-multipliers}\small
\begin{aligned}
  &\int_{{\mathbb{R}^{N}}}(-\Delta)^{\frac{s_1}{2}}u_n
  (-\Delta)^{\frac{s_1}{2}}\varphi+(-\Delta)^{\frac{s_2}{2}}
  u_n(-\Delta)^{\frac{s_2}{2}}\varphi
  -\lambda_{n}u_{n}{\varphi}-\mu|u_{n}|^{q-2}u_{n}{\varphi}
  -|u_{n}|^{p-2}u_{n}{\varphi}dx\\
  &\ \ =o_{n}(1)\|\varphi\|_{H^{s_1}(\mathbb{R}^{N})}
\end{aligned}
\end{equation}
for every $\varphi\in H^{s_1}(\mathbb{R}^{N})$. Taking $\varphi=u_n$, we get
\begin{equation*}
  \lambda_n a^2 =|(-\Delta)^{\frac{s_1}{2}} u_n|_{2}^{2}+|(-\Delta)^{\frac{s_2}{2}} u_n|_{2}^{2}-\mu|u_{n}|_{q}^{q}-|u_{n}|_{p}^{p}+o_{n}(1)
\end{equation*}
and then $\{\lambda_n\}$ is bounded. Hence, up to a subsequence, $\lambda_{n}\to \lambda\in \mathbb{R}$.

\textbf{\emph{Step} 3.  $\lambda<0$.}
First, by $P(u_{n})\to 0$ and the compactness of $H_{r}^{s_1}(\mathbb{R}^N)\hookrightarrow L^s(\mathbb{R}^N)$, one can easily check that $u\not\equiv0$. Then we verify that $\lambda<0$.
 Recalling that $P(u_n)\to0$, we have
\begin{equation}\label{eqn:Pu-0-yingyongdao-1}
  \lambda_n a^2 =\bigg(1-\frac{s_1}{\kappa_p}\bigg)|(-\Delta)^{\frac{s_1}{2}} u_n|_{2}^{2}+\bigg(1-\frac{s_2}{\kappa_p}\bigg)|(-\Delta)^{\frac{s_2}{2}} u_n|_{2}^{2}+
  \mu\bigg(\frac{\kappa_q}{\kappa_p}-1\bigg)|u_{n}|_{q}^{q}+o_{n}(1).
\end{equation}
Note that $ \kappa_p-s_1<0$ and $\kappa_p>\kappa_q$.
Then if $1-\frac{s_2}{\kappa_p}\leq0$, i.e., $p\leq \frac{2N}{N-2s_2}$, it is clear that $\lambda<0$. Now we consider $1-\frac{s_2}{\kappa_p}>0$ and suppose by contradiction that $\lambda\geq0$.
 It
 follows from \eqref{eqn:un-Lagrange-multipliers} that $u$ satisfies
 $$
 (-\Delta)^{s_{1}} u+(-\Delta)^{s_{2}} u=\lambda u+\mu |u|^{q-2}u+|u|^{p-2}u \ \text{in}\;{\mathbb{R}^{N}}.
 $$
 Then we know that
 \begin{equation}\label{eqn:u-pohozaev}
s_{1}|(-\Delta)^{\frac{s_1}{2}} u|_{2}^{2}+s_{2}|(-\Delta)^{\frac{s_2}{2}}u|_{2}^{2}
=\kappa_{p}|u|_{p}^{p}+\mu\kappa_q |u|_{q}^{q}
\end{equation}
and
 \begin{equation}\label{eqn:u-nehari}
|(-\Delta)^{\frac{s_1}{2}} u|_{2}^{2}+|(-\Delta)^{\frac{s_2}{2}}u|_{2}^{2}
=|u|_{p}^{p}+\mu |u|_{q}^{q}+\lambda|u|_2^2.
\end{equation}
On the one hand, since $\kappa_q-s_2<0$ and $1-\frac{s_2}{\kappa_p}>0$, we can infer from \eqref{eqn:u-pohozaev}-\eqref{eqn:u-nehari}
that
$$
\begin{aligned}
(s_1-s_2)|(-\Delta)^{\frac{s_1}{2}} u|_{2}^{2}
&=(\kappa_p-s_2)|u|_p^p+\mu(\kappa_q-s_2)|u|_q^q-\lambda s_2|u|_2^2\\
&\leq (\kappa_p-s_2)|u|_p^p\\
&\leq C_{N,s_1,p}a^{p-\frac{N(p-2)}{2s_1}}|(-\Delta)^{\frac{s_1}{2}} u|_{2}^{\frac{N(p-2)}{2s_1}},
\end{aligned}
$$
where we have used \eqref{eqn:G-N-equality} and the fact that $|u|_2^2\leq a^2$.
Then there holds
\begin{equation}\label{eqn:u-s-1-lower}
|(-\Delta)^{\frac{s_1}{2}} u|_{2}\geq \left( \frac{s_1-s_2}{C_{N,s_1,p}a^{p-\frac{N(p-2)}{2s_1}}}\right)
^{\frac{2s_1}{N(p-2)-4s_1}}.
\end{equation}
On the other hand, using \eqref{eqn:u-pohozaev}-\eqref{eqn:u-nehari} again, we deduce that
$$
\begin{aligned}
(\kappa_p-s_2)|(-\Delta)^{\frac{s_2}{2}} u|_{2}^{2}&=-(\kappa_p-s_1)|(-\Delta)^{\frac{s_1}{2}} u|_{2}^{2}
+\mu(\kappa_p-\kappa_q)|u|_q^q+\lambda\kappa_p|u|_2^2\\
&\geq-(\kappa_p-s_1)|(-\Delta)^{\frac{s_1}{2}} u|_{2}^{2}.
\end{aligned}
$$
By interpolation inequality
\begin{equation}\label{eqn:inter-inequality}
\begin{aligned}
\int_{\mathbb{R}^N}|(-\Delta)^{\frac{s_2}{2}}u|^2 dx\leq\left(\int_{\mathbb{R}^N}|(-\Delta)^{\frac{s_1}{2}}u|^2dx\right)^{\frac{s_2}{s_1}}
\left(\int_{\mathbb{R}^N}|u|^2dx\right)^{\frac{s_1-s_2}{s_1}},
\end{aligned}
\end{equation}
we further obtain that
$$
-(\kappa_p-s_1)|(-\Delta)^{\frac{s_1}{2}} u|_{2}^{2}\leq(\kappa_p-s_2)a^{\frac{2(s_1-s_2)}{s_1}}
|(-\Delta)^{\frac{s_1}{2}} u|_{2}^{\frac{2s_2}{s_1}},
$$
which implies that
\begin{equation}\label{eqn:u-s-1-upper}
|(-\Delta)^{\frac{s_1}{2}} u|_{2}\leq \left( \frac{(\kappa_p-s_2)a^{\frac{2(s_1-s_2)}{s_1}}}{-(\kappa_p-s_1)}
\right)^{\frac{s_1}{2(s_1-s_2)}}=
\left(\frac{\kappa_p-s_2}{s_1-\kappa_p}
\right)^{\frac{s_1}{2(s_1-s_2)}}a.
\end{equation}
Thus, combining \eqref{eqn:u-s-1-lower} and \eqref{eqn:u-s-1-upper}, we can get a contradiction
if
$$\left( \frac{s_1-s_2}{C_{N,s_1,p}a^{p-\frac{N(p-2)}{2s_1}}}\right)
^{\frac{2s_1}{N(p-2)-4s_1}}>\left(\frac{\kappa_p-s_2}{s_1-\kappa_p}
\right)^{\frac{s_1}{2(s_1-s_2)}}a,
$$
that is,
$$
\left(\frac{s_1-s_2}{C_{N,s_1,p}}\right)
^{\frac{2s_1}{N(p-2)-4s_1}}
\left(\frac{s_1-\kappa_p}{\kappa_p-s_2}
\right)^{\frac{s_1}{2(s_1-s_2)}}
>a^{\frac{2(p-2)s_1}{N(p-2)-4s_1}}.
$$
This is the condition $(A_0)$. Thus, $\lambda<0$.

\textbf{\emph{Step} 4. $u_n\to u$ in $H^{s_1}(\mathbb{R}^{N})$}. By weak convergence, \eqref{eqn:un-Lagrange-multipliers} implies that
\begin{equation}\label{eqn:dE-u-lamda-weakly-convergence}
dE(u)\varphi-\lambda\int_{{\mathbb{R}^N}}u\varphi dx=0
\end{equation}
for every $\varphi\in H^{s_1}(\mathbb{R}^{N})$. Choosing $\varphi=u_n - u$ in \eqref{eqn:un-Lagrange-multipliers} and \eqref{eqn:dE-u-lamda-weakly-convergence}, and subtracting, we get
$$
(dE(u_n)-dE(u))[u_n-u]-\lambda\int_{\mathbb{R}^N}|u_n-u|^2dx=o_n (1).
$$
Noting that $u_{n}\to u$ strongly in $L^r(\mathbb{R}^N)$ for $r\in(2,2_{s_{1}}^{\ast})$, we infer that
 $$
|(-\Delta)^{\frac{s_1}{2}} (u_n - u)|_{2}^{2}+|(-\Delta)^{\frac{s_2}{2}} (u_n - u)|_{2}^{2}-\lambda|(u_n - u)|_{2}^{2}=o_n (1),
 $$
 which indicates that $u_n\to u$ in $H^{s_1}(\mathbb{R}^{N})$ due to $\lambda<0$. This completes the proof.
\end{proof}
\subsection{The geometry and existence of a local minimizer for functional $E|_{S_a}$}\label{sub:properties}
\ \\
For the constrained functional $E|_{S_a}$, we can infer from \eqref{eqn:G-N-equality} that
 \begin{equation}\label{eqn:E-mu-G-N-1}\small
 \begin{aligned}
 E(u)\geq\frac{1}{2}|(-\Delta)^{\frac{s_{1}}{2}} u|_{2}^{2}-
\frac{C_{N,s_1,p}}{p}a^{p-\frac{N(p-2)}{2 s_1}}|(-\Delta)^{\frac{s_{1}}{2}}u|_{2}^{\frac{N(p-2)}{2 s_1}}-\frac{\mu C_{N,s_1,q}}{q}a^{q-\frac{N(q-2)}{2 s_1}}|(-\Delta)^{\frac{s_{1}}{2}}u|_{2}^{\frac{N(q-2)}{2 s_1}}
 \end{aligned}
 \end{equation}
for every  $u\in S_a$.
Therefor, to understand the geometry of the functional $E|_{S_a}$, it is
natural to consider the function
 $h:\mathbb{R}^{+}\to\mathbb{R}$
$$
h(t):=\frac{1}{2}t^2 -\frac{C_{N,s_1,p}}{p}a^{p-\frac{N(p-2)}{2 s_1}}t^{\frac{N(p-2)}{2 s_1}}-\frac{\mu C_{N,s_1,q}}{q}a^{q-\frac{N(q-2)}{2 s_1}}t^{\frac{N(q-2)}{2 s_1}}.
$$
Since $2<q< 2+\frac{4s_2}{N}< 2+\frac{4s_1}{N}<p<2_{s_1}^{*}$, then ${\frac{N(q-2)}{2 s_1}}<2<{\frac{N(p-2)}{2 s_1}}$, which implies that $h(t)=0^-$ as $t\rightarrow 0^+$ and $h(t)\rightarrow-\infty$ as $ t\rightarrow \infty$.
Now we give the property of $h(t)$ with respect to parameters $a,\mu>0$, which leads to our hypothesis $(A_1)$
\begin{lemma}\label{lem:h-dayu-0-1}
Under assumption $(A_1)$,  the function $h$ has a local strict minimum at negative level and a global strict maximum at positive level. Moreover, there exist $0<R_0<R_1$, both depending on $a$ and $\mu$, such that $h(R_0)=0=h(R_1)$ and $ h(t)>0$ if and only if $t\in(R_0,R_1).$
\end{lemma}
\begin{proof}
Let
$$
\varphi(t):=\frac{1}{2}t^{\frac{4s_1 -N(q-2)}{2 s_1}} -\frac{C_{N,s_1,p}}{p}a^{p-\frac{N(p-2)}{2 s_1}}t^{\frac{N(p-q)}{2 s_1}},
$$
then we know that $h(t)>0$ if and only if
$$
\varphi(t)>\frac{\mu C_{N,s_1,q}}{q}a^{q-\frac{N(q-2)}{2 s_1}}.
$$
Moreover, the direct calculation implies that $\varphi$ has a unique critical point
 $$
 \bar{t}=\bigg(\frac{p(4s_1-N(q-2))}{2NC_{N,s_1,p}(p-q)a^{\frac{2p s_1-N(p-2)}{2 s_1}}}\bigg)^{\frac{2s_1}{N(p-2)-4s_1}},
 $$
which is a global maximum point at positive level.
Note that
 \begin{equation*}
  \begin{aligned}
 \varphi(\bar{t})&=\frac{1}{2}\left(\frac{p(4s_1-N(q-2))}{2NC_{N,s_1,p}(p-q)a^{\frac{2p s_1-N(p-2)}{2 s_1}}}\right)^{\frac{4s_1-N(q-2)}{N(p-2)-4s_1}}\\
 &\qquad-\frac{C_{N,s_1,p}}{p}a^{p-\frac{N(p-2)}{2 s_1}}\left(\frac{p(4s_1-N(q-2))}{2NC_{N,s_1,p}(p-q)a^{\frac{2p s_1-N(p-2)}{2 s_1}}}\right)^{\frac{N(p-q)}{N(p-2)-4s_1}}\\
 &=a^{\frac{2p s_1-N(p-2)}{2 s_1}\cdot{\frac{N(q-2)-4s_1}{N(p-2)-4s_1}}}
 \left(\frac{p(4s_1-N(q-2))}{2NC_{N,s_1,p}(p-q)}\right)
 ^{\frac{4s_1-N(q-2)}{N(p-2)-4s_1}}\left(\frac{N(p-2)-4s_1}
 {2N(p-q)}\right).
 \end{aligned}
 \end{equation*}
 Therefore, if $(A_1)$ holds, we have $\varphi(\bar{t})>\frac{\mu C_{N,s_1,q}}{q}a^{q-\frac{N(q-2)}{2 s_1}}$, which implies that
 $h$ is positive on an open interval $(R_0, R_1)$.
 This also means that $h$ has a global maximum at positive level in $(R_0, R_1)$. On the other hand, since $h(t)\rightarrow0^-$ as $t\rightarrow 0^+$, we can conclude that there exists a local minimum point at negative level in
 $(0, R_0)$. Moreover, it is easy to check that $h$ has exactly two critical points. Then the proof is complete.
\end{proof}
In the following, we prove that $\mathcal P_0=\emptyset$, which implies that
the manifold $\mathcal{P}$ can be decomposed into two components, that is, $\mathcal{P}_+ $ and $\mathcal{P}_- $.
\begin{lemma}\label{lem:P-natural-subcritical}
If assumption $(A_2)$ holds, then $\mathcal{P}_0=\emptyset$ and $\mathcal{P}$ is a smooth manifold of codimension $2$ in $H^{s_1}(\mathbb{R}^{N})$. In particular, $\mathcal P$ is a natural constraint.
\end{lemma}
\begin{proof}
Assume that there exists $u\in\mathcal{P}_0$, by the definition of $u\in\mathcal{P}_0$, we know that
$P(u)=0$ and $(\Phi_{u})^{\prime\prime}(0)=0$. Then we can deduce that
\begin{equation}\label{eqn:P-u-Phi-combine-1}
s_1(p\kappa_p-2s_1 )|(-\Delta)^{\frac{s_{1}}{2}} u|_{2}^{2}+s_2(p\kappa_p -2s_2 )|(-\Delta)^{\frac{s_{2}}{2}}u|_{2}^{2}=\mu\kappa_q(p\kappa_p -q\kappa_q )|u|_q ^q
\end{equation}
and
\begin{equation}\label{eqn:P-u-Phi-combine-2}
s_1(2s_1-q\kappa_q)|(-\Delta)^{\frac{s_{1}}{2}} u|_{2}^{2}+s_2(2s_2 -q\kappa_q)|(-\Delta)^{\frac{s_{2}}{2}}u|_{2}^{2}=\kappa_p(p\kappa_p -q\kappa_q  )|u|_p ^p.
\end{equation}
Since  $2<q<2+\frac{4s_2}{N}$, it follows from  \eqref{eqn:G-N-equality} that
\begin{equation}\label{eqn:P-u-Phi-combine-3}
|(-\Delta)^{\frac{s_{1}}{2}} u|_{2}^{\frac{4s_1 -N(q-2)}{2 s_1}}\leq\frac{\mu C_{N,s_1,q} \kappa_q(p\kappa_p -q\kappa_q )}{ s_1(p\kappa_p-2s_1) }a^{q-\frac{N(q-2)}{2 s_1}}
\end{equation}
and
\begin{equation}\label{eqn:P-u-Phi-combine-4}
|(-\Delta)^{\frac{s_{1}}{2}} u|_{2}^{\frac{N(p-2)-4s_1 }{2 s_1}}\geq\frac{s_1(2s_1 -q\kappa_q)}{C_{N,s_1,p} \kappa_p(p\kappa_p - q\kappa_q)}a^{\frac{N(p-2)}{2 s_1}-p}.
\end{equation}
From  \eqref{eqn:P-u-Phi-combine-3} and \eqref{eqn:P-u-Phi-combine-4}, we get
\begin{equation*}\small
\begin{aligned}
\bigg(\frac{s_1(p\kappa_p -2s_1)}{\mu C_{N,s_1,q} \kappa_q(p\kappa_p -q\kappa_q )}\bigg)^{\frac{2 s_1}{4s_1 -N(q-2)}} \bigg(\frac{s_1(2s_1 -q\kappa_q)}{C_{N,s_1,p}\kappa_p(p\kappa_p - q\kappa_q )}\bigg)^{\frac{2 s_1}{N(p-2)-4s_1}}
\leq a^{\frac{2qs_1 -N(q-2)}{4s_1 -N(q-2)}+{\frac{2p s_1 -N(p-2)}{N(p-2)-4s_1}}},
\end{aligned}
\end{equation*}
which contradicts with $(A_{2})$.  This proves that $\mathcal{P}_0 =\emptyset$. Moreover, similar to the proof of \cite[Lemma 5.2]{SN2020}, we obtain that $\mathcal{P}$ is a smooth manifold of codimension $2$ in $H^{s_1}(\mathbb{R}^{N})$. In particular,
combining the above facts and Lagrange multipliers rule, it is easy to check that $\mathcal P$ is a natural constraint and then the proof is complete.
\end{proof}

\begin{lemma}\label{lem:t-u-P-Property}
 For every $u\in S_a$, the function $\Phi_u$  has exactly two critical points $\xi_u<t_u\in\mathbb{R}$ and two zeros $c_u <d_u \in\mathbb{R}$ with $\xi_u <c_u < t_u<d_u$. Moreover,
 \begin{itemize}
   \item [$(i)$] $\xi_{u}\star u\in \mathcal{P}_+$ and $t_{u}\star u\in \mathcal{P}_-$, and if $t\star u\in \mathcal{P}$, then either $t=\xi_u$ or $t=t_u$;
   \item [$(ii)$] $|(-\Delta)^{\frac{s_1}{2}}(t\star u)|_{2}\leq R_0$ for every $t\leq c_u$, and
  $$E(\xi_u\star u)=\min\left\{E(t\star u):t\in\mathbb{R}\mathrm{~and~}|(-\Delta)^{\frac{s_1}{2}}(t\star u)|_{2}<R_0 \right\}<0;$$
   \item [$(iii)$] $E\left(t_u\star u\right)=\max\left\{E\left(t\star u\right):t\in\mathbb{R}\right\}>0,$
   and $\Phi_u$ is strictly decreasing and concave on $(t_u, \infty)$. In particular, if $t_u <0$, then $P(u)<0$;
   \item [$(iv)$] The maps $u\in S\mapsto \xi_u\in\mathbb{R}$ and $u\in S\mapsto t_u\in\mathbb{R}$ are of class $C^1$.
 \end{itemize}
\end{lemma}
\begin{proof}
Letting $u\in S_a$, by \eqref{eqn:E-mu-G-N-1}, we have
$$
\Phi_u (t)=E(t\star u)\geq h(|(-\Delta)^{\frac{s_1}{2}}(t\star u)|_{2})= h(e^{s_1 t}|(-\Delta)^{\frac{s_1}{2}}u|_ 2).
$$
 Thus, the $C^2$ function $\Phi_u$  is positive on $(\log(R_0/|(-\Delta)^{\frac{s_1}{2}}u|_ 2),\log(R_1/|(-\Delta)^{\frac{s_1}{2}}u|_ 2))$. Noting that
$2<q<2+\frac{4s_2}{N}< 2+\frac{4s_1}{N}<p$, we can infer that
 $\Phi_u (-\infty)=0^- , \Phi_u (\infty)=-\infty $
  and $\Phi_u ^{\prime}(t)=0$ has at most two
solutions.
  It follows that $\Phi_u$  has exactly two critical points $\xi_u<t_u$, where $\xi_u$ is a local minimum point on $(-\infty, \log(R_0/|(-\Delta)^{\frac{s_1}{2}}u|_ 2))$ at negative level, and $t_u$ is a global maximum point at positive
level.  By Proposition \ref{pro:t-p-uniquess}, we have $t\star u\in \mathcal{P}$ if and only if $\Phi_u ^{\prime}(t)=0$. Then, we conclude that
 $\xi_u \star u, t_u \star u\in \mathcal{P}$, and $t\star u\in\mathcal{P}$ implies $t\in\{\xi_u, t_u\}$. Moreover, since $\mathcal{P}_0 =\emptyset$, we also get that $\xi_u\star u\in \mathcal{P}_+$ and $t_u\star u\in \mathcal{P}_-$.

 It remains to prove $(ii)-(iv)$. The items $(ii)-(iii)$ are obvious due to the above arguments. Indeed, we clearly know that
 $\Phi_u$ has exactly two zeros $c_u <d_u$ with $\xi_u <c_u < t_u<d_u$; and being a $C^2$ function, $\Phi_u$  has exactly two inflection points. This implies that
 $\Phi_u$ is strictly decreasing and concave on $(t_u, \infty)$.
Hence, if $t_u <0$, then $P(u)=\Phi_u^{\prime}(0)<0$.
Now we prove $(iv)$. To this end, we apply the implicit function theorem on the $C^1$ function $\Psi(t, u)=\Phi_u^{\prime}(t)$.
Note that $\Psi(t_u, u)=0$, $\partial_t \Psi(t_u, u)<0$, and it is not possible to pass with continuity from
$\mathcal{P}_+$ to $\mathcal{P}_-$. Thus, we obtain that $u\mapsto t_u$ is $C^1$. The same argument indicates that $u\mapsto \xi_u$ is also $C^1$. The proof is finished.
\end{proof}
Now we denote that
for $k>0$,
$$
A_k:=\{u\in S:|(-\Delta)^{\frac{s_1}{2}}u|_ 2 <k\} \ \text{and}\ \gamma_{\mu}(a):=\inf_{u\in A_{R_0}}E(u).
$$
Then based on the Lemma \ref{lem:t-u-P-Property}, we can directly infer the following result.
\begin{corollary}\label{cor:supE-xiaoyu-0-infE}
The set $\mathcal{P}_+$ is contained in $A_{R_0}=\{u\in S:|(-\Delta)^{\frac{s_1}{2}}u|_ 2 <R_0 \}$ and $\sup_{\mathcal{P}_+}E\leq0\leq \inf_{\mathcal{P}_-}E$.
\end{corollary}
\begin{lemma}\label{lem:gamma-xiaoyu-0}
 It results that $\gamma_{\mu}(a)\in(-\infty, 0)$. Moreover,
 \begin{equation}\label{eqn:gamma-a-properties}
 \gamma_{\mu}(a)=\inf_{\mathcal{P}}E=\inf_{\mathcal{P}_+}E
 \textit{ and } \gamma_{\mu}(a)<\inf_{\overline{A_{R_0}}\setminus A_{R_0-\varrho}}E
\end{equation}
 for $\varrho >0$ small enough.
\end{lemma}
\begin{proof}
For $u\in A_{R_0}$, we know that
$$
E(u)\geq h(|(-\Delta)^{\frac{s_1}{2}}u|_ 2)\geq\min_{t\in[0,R_0]}h(t)>-\infty,
$$
which implies that $\gamma_{\mu}(a)>-\infty$. Moreover, for any $u\in S_a$, we have $|(-\Delta)^{\frac{s_1}{2}}(t\star u)|_ 2< R_0 $ and  $E(t\star u)<0$ for $t<<-1$. Hence, $\gamma_{\mu}(a)\in (-\infty,0)$.

Now, we prove \eqref{eqn:gamma-a-properties}.
By Corollary \ref{cor:supE-xiaoyu-0-infE}, we have $\mathcal{P}_{+}\subset A_{R_0}$, then it is obvious that $\gamma_{\mu}(a)\leq\inf_{\mathcal{P}_+}E$.
 On the other hand, in view of Lemma \ref{lem:t-u-P-Property}, we deduce that if $u\in A_{R_0}$, then $\xi_u\star u \in\mathcal{P}_+\subset A_{R_0}$ and
$$
E(\xi_u\star u)=\min\left\{E(t\star u):t\in\mathbb{R}\mathrm{~and~}|(-\Delta)^{\frac{s_1}{2}}u|_ 2 <R_0\right\}\leq E(u),
$$
which indicates that $\inf_{\mathcal{P}_+}E\leq \gamma_{\mu}(a)<0$.
Thus, $\inf_{\mathcal{P}_+}E= \gamma_{\mu}(a)<0$. Using Corollary \ref{cor:supE-xiaoyu-0-infE} again, we can conclude that
$\inf_{\mathcal{P}_+}E=\inf_{\mathcal{P}}E$.
Finally, by continuity of $h$ there exists $\varrho>0$ such that $h(t)\geq \frac{\gamma_{\mu}(a)}{2}$ if $t\in [ R_0-\varrho,\;R_0 ]$. Therefore, we have
$$
E(u)\geq h(|(-\Delta)^{\frac{s_1}{2}}u|_ 2)\geq\frac{\gamma_{\mu}(a)}{2}>\gamma_{\mu}(a)
$$
for every $u\in S_a$ with $ R_0 -\varrho\leq |(-\Delta)^{\frac{s_1}{2}}u|_ 2 \leq R_0$. This completes the proof.
\end{proof}
\noindent\textbf{Proof of Theorem \ref{Thm:q-p-mass-subcritical-supcritical} (existence of a local minimizer).}
Let $\{v_n\}$ be a minimizing sequence of $\gamma_\mu(a)$ and $|v_n |^{\ast}$ be the symmetric decreasing rearrangement of $v_n$, then it is well-known that $\{|v_n |^{\ast}\}\subset A_{R_0}$ and $E(|v_n |^{\ast})\leq E(v_n)$. Thus, we can assume that
 $v_n \in S_{a,r} $ is radially decreasing for every $n$.
 By Lemma \ref{lem:t-u-P-Property} and Corollary \ref{cor:supE-xiaoyu-0-infE}, there exists $\xi_{v_n}\in\mathbb R$ such that
$\xi_{v_n}\star v_{n}\in\mathcal{P}_+ $, $|(-\Delta)^{\frac{s_1}{2}}(\xi_{v_n}\star v_{n})|< R_0$ and
$$
E\left(\xi_{v_n}\star v_n\right)=\min\left\{E(t\star v_n):t\in\mathbb{R}\mathrm{~and~}|(-\Delta)^{\frac{s_1}{2}}(t\star v_{n})|< R_0\right\}\leq E(v_n).
$$
In this way we obtain a new minimizing sequence $\{w_n:=\xi_{v_n}\star v_{n}\}$ with $w_n\in S_{a,r} \cap \mathcal{P}_{+}$.  From Lemma \ref{lem:gamma-xiaoyu-0}, we know that $|(-\Delta)^{\frac{s_1}{2}}w_n |_ 2 <R_0-\varrho$ for every $n$, then the Ekeland's
variational principle yields that there exists a new minimizing sequence $\{u_n\}\subset A_{R_0}$ for $\gamma_{\mu}(a)$ with the property that $\|u_n -w_n\|_{H^{s_1}(\mathbb{R}^{N})}\to 0$ as $n\to\infty$, which  is also a Palais-Smale sequence for $E(u)$ on $S_a$.
Moreover,
the property $\|u_n -w_n\|_{H^{s_1}(\mathbb{R}^{N})}\to 0$ and  the boundedness of $\{w_n\}$ imply $P(u_n)\to 0$.
 Thus, applying Lemma \ref{lem:mass-supercritical-setting-PS-compact}, we can conculde that, up to a subsequence, $u_n \to \tilde{u}$ in $H^{s_1}(\mathbb{R}^{N})$, $\tilde{u}$ is an interior local minimizer for $E |_{A_{R_0}}$, and solves \eqref{eqn:Mixed-fractioanl-CN} for some $\tilde{\lambda}<0$.
 Moreover, it follows from Lemmas \ref{lem:P-natural-subcritical} and \ref{lem:gamma-xiaoyu-0} that $\tilde u$ is a ground
state for \eqref{eqn:Mixed-fractioanl-CN}, which is  nonnegative and radially decreasing.

\subsection{Existence of a second critical point of mountain pass type for $E |_{S_a}$}

\begin{lemma}\label{lem:Eu-tu-negative-1}
Suppose that $E(u)<\gamma_{\mu}(a)$, then the value $t_u$ defined by Lemma {\rm\ref{lem:t-u-P-Property}} is negative.
\end{lemma}
\begin{proof}
In view of Lemma \ref{lem:t-u-P-Property}, we know that $\xi_u <c_u < t_u<d_u $. Therefore, if $d_u\leq 0$, we have $t_u<0$. Assume by contradiction that $d_u >0$. Then we can infer that
$c_u >0$. Indeed, if $0\in (c_u, d_u)$, then $E(u)=\Phi_u (0)>0$, which contradicts $E(u)<\gamma_{\mu}(a)<0$. Thus, from Lemma \ref{lem:t-u-P-Property}-$(ii)$, we deduce that
\begin{equation*}
\begin{aligned}
\gamma_{\mu}(a)&>E(u)=\Phi_u(0)\geq\inf_{t\in(-\infty,c_u]}\Phi_u(t)\\
&\geq\inf\{E(t\star u):t\in\mathbb{R}\mathrm{~and~}|(-\Delta)^{\frac{s_1}{2}}(t\star v_{n})|_2<R_0\}=E(\xi_u\star u)\geq \gamma_{\mu}(a),
\end{aligned}
\end{equation*}
which is again a contradiction.
\end{proof}

\begin{lemma}\label{lem:sigama-dayu-0-1}
There holds $\inf_{u\in\mathcal{P}_-}E(u)>0.$
\end{lemma}
\begin{proof}
Let $t_{\max}>0$ be the strict maximum of the function $h$ at positive level, see Lemma \ref{lem:h-dayu-0-1} for the properties of $h$. Note that $|(-\Delta)^{\frac{s_1}{2}}(t\star u)|_2=e^{s_1t}|(-\Delta)^{\frac{s_1}{2}}u|_2$ for $t\in\mathbb R$, then
for every $u\in \mathcal{P_-}$, there exists $\tau_u\in\mathbb{R}$ such that $|(-\Delta)^{\frac{s_1}{2}}(\tau_u\star u)|_2=t_{\max}$. On the other hand, for any $u\in \mathcal{P}_-$, it follows from
Lemma \ref{lem:t-u-P-Property} that 0 is the unique strict maximum of the function $\Phi_{u}$. Therefore, from \eqref{eqn:E-mu-G-N-1}, we conclude that
$$
E(u)=\Phi_u(0)\geq\Phi_u(\tau_u)=E(\tau_u\star u)\geq h(|(-\Delta)^{\frac{s_1}{2}}(\tau_u\star u)|_2)=h(t_{\max})>0,
$$
which states that $\inf_{\mathcal{P}_-}E\geq h(t_{max})>0$ and the proof is finished.
\end{proof}

In order to obtain the second critical point of $E|_{S_a}$, it is necessary to denote the tangent space to $S_a$ in $u$ as
$$
T_{u}S_a:=\{v\in H^{s_1}(\mathbb R^N):\int_{\mathbb R^N}uvdx=0\}.
$$
Then the following result holds and its proof is standard, see \cite[Lemma 3.6]{Bartsch-Soave-CV-2019}.
\begin{lemma}\label{lem:Tu-Sa-isomorphism-1}
For $u\in S_a$ and $t\in \mathbb{R}$, the map
$$T_{u}S_a\to T_{t\star u}S_a,\quad\varphi\mapsto t\star\varphi $$
is a linear isomorphism with inverse $\psi\mapsto(-t)\star\psi.$
\end{lemma}

\noindent\textbf{Proof of Theorem \ref{Thm:q-p-mass-subcritical-supcritical} (existence of a second critical point for $E|_{S_a}$).}
Inspired by \cite{JEAN-NA-1997}, we
consider the augmented functional $\tilde{E}:\mathbb{R}\times H^{s_1}(\mathbb{R}^N) \to \mathbb{R}$ defined by
\begin{equation}\label{eqn:E-tilde-u-definition}\small
\begin{aligned}
\tilde{E}(t,u)&:=E(t\star u)\\
&=\frac{e^{2s_1 t}}2\int_{\mathbb{R}^N}|(-\Delta)^{\frac{s_1}{2}} u|^2dx+\frac{e^{2s_2 t}}2\int_{\mathbb{R}^N}|(-\Delta)^{\frac{s_2}{2}} u|^2dx\\
&\ \ -\mu\frac{e^{\kappa_{q}qt}}{q}\int_{\mathbb{R}^N}|u|^qdx
-\frac{e^{\kappa_{p}pt}}{p}\int_{\mathbb{R}^N}|u|^pdx
\end{aligned}
\end{equation}
and look at the restriction $\tilde{E}|_{\mathbb{R}\times S_{a,r}}$. Obviously, $\tilde{E}$ is of class $C^1$ and a Palais-Smale sequence for $\tilde{E}|_{\mathbb{R}\times S_{a,r}}$ is also a Palais-Smale sequence for $\tilde{E}|_{\mathbb{R}\times S_a}$.

 Denoting by $E^c$ the closed sublevel set $\{u\in S_{a}:E(u)\leq c\}$, we introduce the minimax class
 \begin{equation}\label{eqn:hua-F-definition-1}
 \mathcal{F}:=\left\{f=(\alpha,\beta)\in C([0,1], \mathbb{R}\times S_{a,r}):f(0)\in(0,\mathcal{P}_+),f(1)\in(0,E^{2\gamma_{\mu}(a)})
 \right\}
 \end{equation}
 with associated minimax level
 $$
 \sigma_{\mu}(a)=\inf_{f\in\mathcal{F}}\max_{(t,u)\in f([0,1])}\tilde{E}(t,u).
 $$
Let $u\in S_{a,r}$, then it follows from \eqref{eqn:s-u-continous-Hs1} and Lemma \ref{lem:t-u-P-Property} that there exists $\xi_0 >>1$ such that
\begin{equation}\label{eqn:f-u-tau-1}
f_u:\tau\in[0,1]\mapsto(0,((1-\tau)\xi_u+\tau \xi_0)\star u)\in\mathbb{R}\times S_{a,r}
\end{equation}
 is a path in $\mathcal{F}$, which means that $\mathcal{F}\not=\emptyset$.
Now we claim that
\begin{equation}\label{eqn:f-Hua-shuyu-P-1}
 \text{for every}\; f\in\mathcal{F},\ \text{there exists }\;\tau_f\in(0,1)\text{ such that }\alpha(\tau_f)\star\beta(\tau_f)\in\mathcal{P}_-.
\end{equation}
Indeed, since $f(0)=(\alpha(0),\beta(0))\in(0,\mathcal{P}_+)$, we can infer from Proposition \ref{pro:t-p-uniquess} and Lemma \ref{lem:t-u-P-Property} that
$$
t_{\alpha(0)\star\beta(0)}=t_{\beta(0)}>\xi_{\beta(0)}=0.
$$
On the other hand, since $E(\beta(1))=\tilde{E}(f(1))\leq2\gamma_{\mu}(a)<\gamma_{\mu}(a)$, then it follows from Lemma \ref{lem:Eu-tu-negative-1} that
 $$
 t_{\alpha(1)\star\beta(1)}=t_{\beta(1)}<0.
 $$
 Noting that $t_{\alpha(\tau)\star\beta(\tau)}$  is continuous in $\tau$, we derive that there exists $\tau_{f}\in(0,1)$ such that $t_{\alpha(\tau_f)\star\beta(\tau_f)}=0$, that is, the claim \eqref{eqn:f-Hua-shuyu-P-1} is true.

In light of the claim above, we can get that
\begin{equation}\label{eqn:sigma-a=P-}
 \sigma_{\mu}(a)=\inf_{\mathcal{P}_- \cap S_{a,r}} E.
\end{equation}
In fact, by \eqref{eqn:f-Hua-shuyu-P-1}, we have
 $$
 \max_{f([0,1])}\tilde{E}\geq\tilde{E}(f(\tau_{f}))=E(\alpha(\tau_{f})\star\beta(\tau_{f}))\geq\inf_{\mathcal{P}_{-}\cap S_{r}}E,
 $$
and then $\sigma_{\mu}(a)\geq\inf_{\mathcal{P}_- \cap S_{a,r}} E$.
For the reverse inequality, it suffices to consider
 $u\in\mathcal{P}_- \cap S_{a,r}$ and the path $f_u\in\mathcal{F}$ defined by
\eqref{eqn:f-u-tau-1}, then we can determine that
$$
E\left(u\right)=\tilde{E}\left(0,u\right)=\max_{f_u\left([0,1]\right)}\tilde{E}\geq\sigma_{\mu}(a),
$$
which yields that $\inf_{\mathcal{P}_-\cap S_{a,r}}E\geq\sigma_{\mu}(a)$.

Combining \eqref{eqn:sigma-a=P-}, Corollary \ref{cor:supE-xiaoyu-0-infE} and Lemma \ref{lem:sigama-dayu-0-1}, we deduce that
 \begin{equation}\label{eqn:sigma-mu-a-1}
 \sigma_{\mu}(a)=\inf_{\mathcal{P}_-\cap S_{a,r}}E>0\geq\sup_{(\mathcal{P}_+\cup E^{2\gamma_{\mu}(a)})\cap S_{a,r}}E=\sup_{((0,\mathcal{P}_+)\cup(0,E^{2\gamma_{\mu}(a)}))\cap S_{a,r}}\tilde{E}.
 \end{equation}
Using the terminology in \cite[Section 5]{Ghoussoub-1993},
this means that $\{f\left([0,1]\right):f\in\mathcal{F}\}$ is a homotopy stable
family of compact subsets of $\mathbb{R}\times S_{a,r} $ with extended closed boundary $(0,\mathcal{P}_+)\cup(0,E^0)$, and the superlevel set $\{\tilde{E}\geq\sigma_{\mu}(a)\}$ is a dual set, in the sense that assumptions $(F'1)$ and $(F'2)$
in \cite[Theorem 5.2]{Ghoussoub-1993} are satisfied. Let
$\{f_n=(\alpha_n,\beta_n)\}\subset\mathcal{F} $ be the
 any minimizing sequence for $\sigma_{\mu}(a)$ with the property that $\alpha_n\equiv0$ and $\beta_n(\tau)\geq0$ a.e. in $ \mathbb{R}^N$  for every $\tau\in [0,1]$, then \cite[Theorem 5.2]{Ghoussoub-1993} implies that
  there exists a Palais-Smale sequence $\{(t_n,w_n)\}\subset\mathbb{R}\times S_{a,r}$ for $\tilde{E}|_{\mathbb{R}\times S_{a,r}}$ at level
$\sigma_{\mu}(a)$, that is
\begin{equation}\label{eqn:partial-E-t-1}
\partial_t\tilde{E}(t_n,w_n)\to0\text{ and } \|\partial_u\tilde{E}(t_n,w_n)\|_{(T_{w_n}S_r)^*}\to0
\text{ as }n\to\infty,
\end{equation}
 with the additional property that
\begin{equation}\label{eqn:tn-wn-beta-1}
|t_n|+\mathrm{dist}_{H^{s_1}(\mathbb{R}^N)}(w_n,\beta_n([0,1]))
\to0\text{ as }n\to\infty.
\end{equation}
Due to the definition of $\tilde E$ (see \eqref{eqn:E-tilde-u-definition}) and \eqref{eqn:partial-E-t-1}, we have $P(t_n\star w_n)\to 0$ and
\begin{equation*}
\begin{aligned}
&e^{2s_1 t_n}\int_{\mathbb{R}^N}(-\Delta)^{\frac{s_1}{2}} w_n (-\Delta)^{\frac{s_1}{2}}\varphi dx +e^{2s_2 t_n}\int_{\mathbb{R}^N}(-\Delta)^{\frac{s_2}{2}} w_n  (-\Delta)^{\frac{s_2}{2}}\varphi dx\\
&\ \ -\mu e^{\kappa_{q}qt}\int_{\mathbb{R}^N}|w_n |^{q-2}w_n \varphi dx-e^{\kappa_{p}pt}\int_{\mathbb{R}^N}|w_n |^{p-2}w_n \varphi dx=o_{n}(1)\|\varphi\|_{H^{s_1}(\mathbb{R}^N)}
\end{aligned}
\end{equation*}
for every $\varphi\in T_{w_n}S_{a,r}$. Moreover, by \eqref{eqn:tn-wn-beta-1}, we know that $\{t_n\}$  is uniformly bounded and then
\begin{equation}\label{eqn:dE-tn-wn-1}
dE(t_n\star w_n)[t_n\star\varphi]=o_n(1)\|\varphi\|_{H^{s_1}(\mathbb{R}^{N})}
=o_{n}(1)\|t_n\star\varphi\|_{H^{s_1}(\mathbb{R}^{N})}.
\end{equation}
Let $u_n :=t_n \star w_n$, then Lemma \ref{lem:Tu-Sa-isomorphism-1} and \eqref{eqn:dE-tn-wn-1} signify that $\{u_n\}\subset S_{a,r}$  is a Palais-Smale sequence for $E|_{S_a}$ at level $\sigma_{\mu}(a)>0$, with $P(u_n)\to 0$. By Lemma \ref{lem:mass-supercritical-setting-PS-compact} and \eqref{eqn:tn-wn-beta-1}, we conclude that, up to a subsequence, $u_n \to\hat{u}$ in $H^{s_1}(\mathbb{R}^N)$, where $\hat{u}\in S_a$ is a nonnegative radial solution of \eqref{eqn:Mixed-fractioanl-CN} for some $\hat{\lambda}<0$.

\section{Sobolev critical leading term with a $L^2$-subcritical perturbation}\label{sec:sobolev-critical}
In this section, we consider $2<q< 2+\frac{4s_2}{N}<p=2_{s_1}^*$
and always assume that $N\geq2$ and $0<s_2<s_1<1$.
Moreover, we establish the existence of a local minimizer for the
constrained functional and determine this minimizer is a ground state of problem \eqref{eqn:Mixed-fractioanl-CN}, that  is, Theorem \ref{Thm:q-mass-subcritical-p-sobolev-critical} is obtained.

Recall the fractional Sobolev inequality \eqref{eqn:sobolev-critical-inequality}, i.e.,
\begin{equation*}
\mathcal{S} |u|_{2_{s_1}^{*}}^2\leq |(-\Delta)^{\frac{s_1}{2}}u|_2 ^2,\quad\forall u\in H^{s_1}(\mathbb{R}^N).
\end{equation*}
Then, similar to \eqref{eqn:E-mu-G-N-1}, we have
\begin{equation}\label{eqn:E-mu-dayu-h}
E(u)\geq |(-\Delta)^{\frac{s_1}{2}}u|_2 ^2 h(a,|(-\Delta)^{\frac{s_1}{2}}u|_2)\ \mbox{for every}\ u\in S_a,
\end{equation}
where
 $$
h(a,\rho):=
\frac{1}{2} -\frac{1}{2_{s_1}^{*}}\frac{1}{{{\mathcal{S} }^{\frac{2_{s_1}^{*}}{2}}}}\rho^{\frac{4s_1 }{N-2s_1}}-\frac{\mu C_{N,s_1,q}}{q}a^{\frac{2N-q(N-2s_1)}{2 s_1}}\rho^{\frac{N(q-2)-4s_1 }{2 s_1}}.
$$
\begin{lemma}\label{lem:h-global-maximum}
For each $a>0$, the function $\rho\mapsto h(a,\cdot)$ has a unique global maximum and there exists $a_0=a_0(\mu)$ such that the maximum value
satisfies
\begin{equation*}
\begin{cases}
    \max\limits_{\rho>0}h(a,\rho)>0&if \ a<a_0;\\
    \max\limits_{\rho>0}h(a,\rho)=0&if \ a=a_0;\\
    \max\limits_{\rho>0}h(a,\rho)<0&if \ a>a_0,\\
\end{cases}
\end{equation*}
where
\begin{equation}\label{eqn:definition-a-0-1}
 a_0 :=\mu^{-{\frac{2s_1}{2qs_1 -N(q-2)}}}\left(\frac{1}{2K_{0}}\right)^{\frac{N}{4s_1 }}
\end{equation}
with

\begin{equation*}
\begin{aligned}
K_0:&=
\frac{1}{2_{s_1}^{*}}\frac{1}{{{\mathcal{S} }^{\frac{2_{s_1}^{*}}{2}}}}\left(\frac{C_{N,s_1,q}(4s_1 -N(q-2))}{2s_1q}\cdot\frac{N \mathcal{S}^{\frac{2_{s_1}^{*}}{2}}}{2s_1}\right)^{\frac{8 s_1 ^{2}}{N(2qs_1 -N(q-2))}}\\
&\ \ +\frac{ C_{N,s_1,q}}{q}\left(\frac{ C_{N,s_1,q}(4s_1 -N(q-2))}{2s_1q}\cdot\frac{N \mathcal{S}^{\frac{2_{s_1}^{*}}{2}}}{2s_1}\right)^{\frac{(N-2s_1 )(N(q-2)-4s_1 )}{N(2qs_1 -N(q-2))}}.
\end{aligned}
\end{equation*}
\end{lemma}
\begin{proof}
Note that $\frac{N(q-2)-4s_1 }{2 s_1}<0$. Then
based on simple calculations, we can infer that $h(a,\rho)$  has a unique critical point $\rho_{0}$, which is a
global maximum and
\begin{equation}\label{eqn:t-a-A-1}\small
\rho_{a}=\left(\frac{\mu C_{N,s_1,q}(4s_1 -N(q-2))}{2s_1q}\cdot\frac{N \mathcal{S}^{\frac{2_{s_1}^{*}}{2}}}{2s_1}\right)^{\frac{(N-2s_1 )2s_1}{N(2qs_1 -N(q-2))}}a^{\frac{N-2s_1}{N}}
:=A_{\mu}^{\frac{(N-2s_1 )2s_1}{N(2qs_1 -N(q-2))}}a^{\frac{N-2s_1}{N}}.
\end{equation}
It results that the maximum value
\begin{equation*}
\begin{aligned}
h(a,\rho_a )&=\frac{1}{2}-\frac{1}{2_{s_1}^{*}}\frac{1}{{{\mathcal{S} }^{\frac{2_{s_1}^{*}}{2}}}}A_\mu^{\frac{8 s_1 ^{2}}{N(2qs_1 -N(q-2))}}a^{\frac{4s_1 }{N}}-\frac{\mu C_{N,s_1,q}}{q}A_\mu^{\frac{(N-2s_1 )(N(q-2)-4s_1 )}{N(2qs_1 -N(q-2))}}a^{\frac{4s_1 }{N}}\\
&:=\frac{1}{2}-K_\mu a^{\frac{4s_1 }{N}}.
\end{aligned}
\end{equation*}
This implies that the proof can be finished if we choose
$$
a_0 =\left(\frac{1}{2K_{\mu}}\right)^{\frac{N}{4s_1 }}=\mu^{-{\frac{2s_1}{2qs_1 -N(q-2)}}}\left(\frac{1}{2K_{0}}\right)^{\frac{N}{4s_1 }},
$$
 where we have used the fact that
$$
\frac{(N-2s_1 )(N(q-2)-4s_1 )}{N(2qs_1 -N(q-2))}+1=\frac{8 s_1 ^{2}}{N(2qs_1 -N(q-2))}.
$$
\end{proof}
\begin{lemma}\label{lem:h-a2-t2}
Let $(a_1,\rho_1)\in(0,\infty)\times(0,\infty)$ be such that $h(a_1 ,\rho_1)\geq0$. Then for any $a_2 \in(0, a_1 ]$, it follows that
$$
h(a_2 ,\rho_2)\geq 0\ \text{if}\ \rho_2 \in \left[\frac{a_2}{a_1}\rho_1, \rho_1\right].
$$
\end{lemma}
\begin{proof}
Obviously, the function $a\to h(\cdot,\rho)$ is non-increasing, then we have that
\begin{equation}\label{eqn:h-a2-a1}
 h(a_2 ,\rho_1)\geq h(a_1 ,\rho_1)\geq 0.
\end{equation}
Moreover, due to $a_2\leq a_1$, the direct calculation implies that
\begin{equation}\label{eqn:h-a2-t1}
  h(a_2 ,\frac{a_2 }{a_1}\rho_1) \geq h(a_1 ,\rho_1)\geq 0.
\end{equation}
By Lemma \ref{lem:h-global-maximum}, we know that $h(a_2,\rho)$ has a unique critical point with respect
to $\rho>0$, which is a unique global maximum.
Thus, combining \eqref{eqn:h-a2-a1} and \eqref{eqn:h-a2-t1}, we can conclude that
$$
h(a_2 ,\rho_2)\geq 0\ \text{if}\ \rho_2 \in \left[\frac{a_2}{a_1}\rho_1, \rho_1\right].
$$
The proof is complete.
\end{proof}

Now let $a_0 >0$ be given by \eqref{eqn:definition-a-0-1} and $\rho_0:= \rho_{a_0}>0$ be determined by \eqref{eqn:t-a-A-1}.
%In particular, we know that
%$$
%\frac{(N-2s_1 )(N(q-2)-4s_1 )}{N(2qs_1 -N(q-2))}+1=\frac{8 s_1 ^{2}}{N(2qs_1 -N(q-2))},
%$$
%then $a<a_0$ is equivalent to
%$$
%a\mu^{\frac{2s_1}{N(2qs_1 -N(q-2))}}<\left(\frac{1}{2K_0}\right)^{\frac{N}{4s_1 }},
%$$
%that is, our condition $(\bar A_0)$.
By Lemmas \ref{lem:h-global-maximum} and \ref{lem:h-a2-t2}, we have $h(a_0, \rho_0)=0$ and $h(a, \rho_0)>0$ for all $a\in(0, a_0)$. Moreover, the inequality \eqref{eqn:E-mu-dayu-h} implies that we can define
$$A_{a,\rho_0}:=\{u\in S_a:|(-\Delta)^{\frac{s_1}{2}}u|_2 <\rho_{0}\}
$$
 and consider the following local minimization problem: for any $a\in (0, a_0)$,
\begin{equation}\label{eqn:definition-gamma-c-1}
  \overline{\gamma_{\mu}}(a):=\inf_{u\in A_{a,\rho_0}} E(u).
\end{equation}
\begin{lemma}\label{lem:gamma-a-continuous-1}
For any $a\in(0, a_0 )$, there hold that
\begin{itemize}
  \item [$(i)$]
  $$
 \overline{\gamma_{\mu}}(a)= \inf_{u\in A_{a,\rho_0}} E(u)<0<\inf_{u\in \partial A_{a,\rho_0}} E(u);
  $$
  \item [$(ii)$] $a\mapsto \overline{\gamma_{\mu}}(a)$ is a continuous mapping;
  \item [$(iii)$] for all $c\in(0, a)$,  $\overline{\gamma_{\mu}}(a)\leq \overline{\gamma_{\mu}}(c)+\overline{\gamma_{\mu}}(\sqrt{a^2-c^2})$ and if $\overline{\gamma_{\mu}}(c)$ is
  reached then the inequality is strict.
\end{itemize}
\end{lemma}

\begin{proof}
$(i)$ For any $u\in \partial A_{a,\rho_0}$, we have $|(-\Delta)^{\frac{s_1}{2}}u|_2 ^2=\rho_{0}^{2}$. It follows from \eqref{eqn:E-mu-dayu-h} that
$$
E(u)\geq |(-\Delta)^{\frac{s_1}{2}}u|_2 ^2 h(a, |(-\Delta)^{\frac{s_1}{2}}u|_2)=\rho_{0}^{2}h(a, \rho_{0})>0.
$$
Furthermore, from the definition of $\Phi_{u}(t)$,  we get $\Phi_{u}(t)\to 0^-$ as $t\to -\infty$ since $\kappa_q q<2s_2$. Therefore, there exists $t_0 <<-1$ such that $|(-\Delta)^{\frac{s_1}{2}}(t_0\star u)|_2 ^2 = e^{2s_1 t_0}|(-\Delta)^{\frac{s_1}{2}}u|_2 ^2< \rho_{0}^{2}$ and $E(t_0\star u)=\Phi_{u}(t_0)<0$. This implies that $\overline{\gamma_{\mu}}(a)<0$.\par
$(ii)$ For any $a\in(0, a_0)$,  let $\{a_n\}\subset(0, a_0)$ be such that $a_n \to a$ as $n\rightarrow\infty$. Then from the definition of $\overline{\gamma_{\mu}}(a_n)$ and item $(i)$, we deduce that for any $\varepsilon>0$  sufficiently small, there exists $u_n \in A_{a_n,\rho_0}$ such that
\begin{equation}\label{eqn:E-u-n-gamma-1}
  E(u_n)\leq \overline{\gamma_{\mu}}(a_n)+\varepsilon\ \text{and}\  E(u_n)<0.
\end{equation}
Setting $v_n:=\frac{a}{a_n } u_n \in S_a$, we claim  that $v_n\in A_{a,\rho_0}$. Indeed, if $a_n \geq a$, then
$$
|(-\Delta)^{\frac{s_1}{2}}v_n |_2 ^2 =\bigg(\frac{a}{a_n }\bigg)^{2} |(-\Delta)^{\frac{s_1}{2}}u_n |_2 ^2<\rho_0 ^{2}.
$$
If $a_n <a$, by Lemma \ref{lem:h-a2-t2}, we can get that $h(a_n, \rho)\geq 0$ for any $\rho \in\left[\frac{a_n }{a}\rho_0,\;\rho_0\right]$. Taking account of \eqref{eqn:E-mu-dayu-h} and \eqref{eqn:E-u-n-gamma-1}, we have $h(a_n, |(-\Delta)^{\frac{s_1}{2}}u_n |_2)<0$, which implies that $|(-\Delta)^{\frac{s_1}{2}}u_n |_2< \frac{a_n }{a}\rho_0$
and then
$$
|(-\Delta)^{\frac{s_1}{2}}v_n |_2 ^2<\bigg(\frac{a}{a_n }\bigg)^{2}\bigg(\frac{a_n }{a}\rho_0\bigg)^2=\rho_0 ^2.
$$
Namely, the claim holds.
Hence, we obtain that
$$
\overline{\gamma_{\mu}}(a)\leq E(v_n)=E(u_n)+[E(v_n)-E(u_n)],
$$
where
\begin{equation*}
\begin{aligned}
  E(v_n)-E(u_n)&=\frac{1}{2}\bigg[\bigg(\frac{a}{a_n }\bigg)^{2}-1\bigg]\left(|(-\Delta)^{\frac{s_1}{2}}u_n |_2 ^2+|(-\Delta)^{\frac{s_2}{2}}u_n |_2 ^2\right)\\
  &\ \ -\frac{\mu}{q}\bigg[\bigg(\frac{a}{a_n }\bigg)^{q}-1\bigg]|u_n |_{q}^{q}-\frac{1}{2_{s_1}^{*}}\bigg[\bigg(\frac{a}{a_n }\bigg)^{2_{s_1}^{*}}-1\bigg]|u_n|_{2_{s_1}^{*}}^{2_{s_1}^{*}}.
\end{aligned}
\end{equation*}
Since $\{
u_n\}$ is bounded in $H^{s_1}(\mathbb R^N)$, then as $n\to \infty$, we have
\begin{equation}\label{eqn:gamma-c-E-yn-1}
  \overline{\gamma_{\mu}}(a)\leq E(v_n)=E(u_n)+o_{n}(1).
\end{equation}
Combining \eqref{eqn:E-u-n-gamma-1} and \eqref{eqn:gamma-c-E-yn-1}, we conclude that
$$
\overline{\gamma_{\mu}}(a)\leq\overline{\gamma_{\mu}}(a_n)+\varepsilon+o_{n}(1).
$$
On the other hand, let $u\in A_{a,\rho_0}$ be such that
$$
E(u)\leq  \overline{\gamma_{\mu}}(a)+\varepsilon\ \text{and} \  E(u)<0.
$$
Similarly, setting $w_n=\frac{a_n}{a} u\in S_{a_n}$, we can infer that $|(-\Delta)^{\frac{s_1}{2}}w_n |_2<\rho_0$ and
$$
\overline{\gamma_{\mu}}(a_n)\leq \overline{\gamma_{\mu}}(a)+\varepsilon+o_{n}(1).
$$
Due to $\varepsilon>0$ is arbitrary, we conclude that $\overline{\gamma_{\mu}}(a_n)\to \overline{\gamma_{\mu}}(a)$ as $n\rightarrow\infty$.
The item $(ii)$ is proved.

$(iii)$ Noting that $c<a<a_0$, we can deduce that for any $\varepsilon>0$  sufficiently small, there exists $u_n \in A_{c,\rho_0}$ such that
$$
  E(u_n)\leq \overline{\gamma_{\mu}}(c)+\varepsilon\ \text{and}\  E(u_n)<0.
$$
Similar to the proof of $(ii)$, we have
$|(-\Delta)^{\frac{s_1}{2}}u_n |_2 ^2<\left(\frac{c}{a}\rho_0\right) ^{2}.$
Let $\theta\in (1,\frac{a}{c}]$ and $v_n:=\theta u_n$, then $|v_n|_2^2=\theta^2 c^2$ and
$|(-\Delta)^{\frac{s_1}{2}}v_n |_2 ^2 =\theta^2|(-\Delta)^{\frac{s_1}{2}}u_n |_2 ^2<\rho_0^2$, which imply that $v_n\in A_{\theta c,\rho_0}$.
Thus, we can conclude that
$$
\overline{\gamma_\mu}(\theta c)\leq E(v_n)<\theta^2 E(u_n)\leq \theta^2(\overline{\gamma_{\mu}}(c)+\varepsilon),
$$
which yields that
\begin{equation*}
\overline{\gamma_\mu}(\theta c)\leq \theta^2\overline{\gamma_{\mu}}(c)
\end{equation*}
and the inequality is strict if $\overline{\gamma_{\mu}}(c)$ reached.
It follows that
$$
\begin{aligned}
\overline{\gamma_{\mu}}(a)&=\frac{a^2-c^2}{a^2}
\overline{\gamma_{\mu}}(a)+\frac{c^2}{a^2}\overline{\gamma_{\mu}}(a)\\
&=\frac{a^2 -c^2}{a^2}\overline{\gamma_\mu}\left(\frac{a}{\sqrt{a^2-c^2}}
\sqrt{a^2-c^2}\right)+
\frac{c^2}{a^2}\overline{\gamma_\mu}\left(\frac{a}{c} c\right)\\
&\leq \overline{\gamma_{\mu}}(\sqrt{a^2-c^2})+\overline{\gamma_{\mu}}(c)
\end{aligned}
$$
and the inequality is also strict if $\overline{\gamma_{\mu}}(c)$ reached. The proof is complete.
\end{proof}

In the following, we will prove that $\overline{\gamma_{\mu}}(a)$ is achieved. Before that, we  exclude the possibility of vanishing  occurs in the following sense.
\begin{lemma}\label{lem:vanishing-u-n-1}
For any $ a\in (0, a_0 )$, let $\{u_n\}\subset A_{a,\rho_0}$ be such that $E(u_n )\to \overline{\gamma_{\mu}}(a)$ as $n\rightarrow\infty$. Then there exist $\beta_0 >0$ and a sequence $\{y_n\}\subset \mathbb{R}^{N}$ such that
\begin{equation}\label{eqn:un-yn-R-beta-1}
\int_{B(y_n,R)} |u_n|^2 \geq \beta_0>0\ \text{for some}\ R>0.
\end{equation}
\end{lemma}
\begin{proof}
We assume by contradiction that \eqref{eqn:un-yn-R-beta-1} does not hold. Since $\{u_n\}$ is bounded in $H^{s_1}(\mathbb{R}^{N})$, then by  \cite[Lemma 2.2]{Felmer2012}, we infer that $|u_n|_q ^{q}\to 0$ as $n\to \infty$. Thus, it follows from \eqref{eqn:sobolev-critical-inequality} that
\begin{equation*}
\begin{aligned}
E(u_n)&=\frac{1}{2}|(-\Delta)^{\frac{s_1}{2}}u_n |_2^2+\frac{1}{2}|(-\Delta)^{\frac{s_2}{2}}u_n |_2^2-\frac{1}{2_{s_1}^{*}}|u_n|_{2_{s_1}^{*}}^{2_{s_1}^{*}}+o_{n}(1)
\\
&\geq|(-\Delta)^{\frac{s_1}{2}}u_n |_2^2\bigg[\frac{1}{2}-\frac{1}{2_{s_1}^{*}}\frac{1}
{\mathcal{S}_{1}^{\frac{2_{s_1}^{*}}{2}}}|(-\Delta)^{\frac{s_1}{2}}u_n |_2^{\frac{4s_1}{N-2s_1}}\bigg]+o_{n}(1)\\
&\geq |(-\Delta)^{\frac{s_1}{2}}u_n |_2^2\bigg[\frac{1}{2}-\frac{1}{2_{s_1}^{*}}\frac{1}
{\mathcal{S}_{1}^{\frac{2_{s_1}^{*}}{2}}}\rho_{0}^
{\frac{4s_1}{N-2s_1}}\bigg]+o_{n}(1)
\\
&=|(-\Delta)^{\frac{s_1}{2}}u_n |_2^2\left(
\frac{\mu C_{N,s_1,q}}{q}a_{0}^{\frac{2N-q(N-2s_1)}{2 s_1}}\rho_{0}^{\frac{N(q-2)-4s_1 }{2 s_1}}\right)+o_n(1)\\
&:=\alpha_{0}|(-\Delta)^{\frac{s_1}{2}}u_n |_2^2+o_n(1),
\end{aligned}
\end{equation*}
where we have used the fact that $h(a_0, \rho_0)=0$.
This implies that
 $E(u_n)\geq o_{n}(1)$, which contradicts $\overline{\gamma_{\mu}}(a)<0$ and the proof is complete.
\end{proof}

\begin{lemma}\label{lem:hua-M-c-shouxlianxing}
For any $a\in (0,a_0)$, it results that $\overline{\gamma_{\mu}}(a)$ is achieved.
\end{lemma}
\begin{proof}
Let $\{u_n\}\subset A_{a,\rho_0}$ be the minimizing sequence such that $E(u_n )\to \overline{\gamma_{\mu}}(a)$ as $n\rightarrow\infty$. Then
 Lemma \ref{lem:vanishing-u-n-1} indicates that there exists a
sequence $\{y_n\}\subset \mathbb{R}^{N}$ such that
$$
u_n(x-y_n)\rightharpoonup u_a\neq0\quad\text{in }H^{s_1}(\mathbb{R}^{N}).
$$
 In the following, we prove that $w_{n}(x):=u_{n}(x-y_{n})-u_{a}(x)\to0\text{ in }H^{s_1}(\mathbb{R}^{N})$, which shows that $\overline{\gamma_{\mu}}(a)$ is achieved by $u_a\not=0$.  By the Br\'{e}zis-Lieb type lemma, we have
\begin{equation}\label{eqn:E-un-x-yn-1}
|w_n|_{2}^{2}=a^2-|u_a|_{2}^{2}
+o_{n}(1),\ |(-\Delta)^{\frac{s_1}{2}}w_n|_{2}^{2}=|(-\Delta)^{\frac{s_1}{2}} u_n|_{2}^{2}-|(-\Delta)^{\frac{s_1}{2}} u_a|_{2}^{2}+o_{n}(1)
\end{equation}
and
\begin{equation*}
%&|(-\Delta)^{\frac{s_1}{2}}w_n|_{2}^{2}=|(-\Delta)^{\frac{s_1}{2}} u_n|_{2}^{2}-|(-\Delta)^{\frac{s_1}{2}} u_a|_{2}^{2}+o_{n}(1)\\
%&|(-\Delta)^{\frac{s_2}{2}}w_n|_{2}^{2}=|(-\Delta)^{\frac{s_2}{2}} u_n|_{2}^{2}-|(-\Delta)^{\frac{s_2}{2}} u_a|_{2}^{2}+o_{n}(1)\\
E(w_n)+E(u_a)=E(u_n (x-y_n ))+o_{n}(1),
\end{equation*}
which imply that
$|u_a|_2^2\leq a^2$ and
\begin{equation}\label{eqn:E-un-x-yn-2}
E(u_n)=E(u_n(x-y_n))=E(w_n)+E(u_a)+o_n(1).
\end{equation}
Now we claim that $|u_a|_2^2=a^2$, which means that $|w_n|_2^2=o_n(1)$.
Assume by contradiction that $\tilde a:=|u_a|_2\in (0,a)$. In view of \eqref{eqn:E-un-x-yn-1}, for $n$ large enough, we have $|w_n|_2 ^{2}\leq a^2$ and $|(-\Delta)^{\frac{s_1}{2}}w_n|_{2}^{2}\leq |(-\Delta)^{\frac{s_1}{2}}u_n|_{2}^{2}<\rho_0 ^{2}$. Then we obtain that $w_n \in A_{|w_n|_{2},\rho_0}$ and $E(w_n)\geq \overline{\gamma_{\mu}}(|w_n|_{2})$. By \eqref{eqn:E-un-x-yn-2}, we infer that
$$
\overline{\gamma_{\mu}}(a)=E(w_n)+E(u_a)+o_n(1)\geq \overline{\gamma_{\mu}}(|w_n|_{2})+E(u_a)+o_n(1),
$$
which, together with Lemma \ref{lem:gamma-a-continuous-1}-$(ii)$, yields that
\begin{equation}\label{eqn:gamma-a-dayu-gamma-1}
\overline{\gamma_{\mu}}(a)\geq \overline{\gamma_{\mu}}(\sqrt{a^2-\tilde a^2})+E(u_a).
\end{equation}
Note that $u_a\in A_{\tilde a,\rho_0}$, then $E(u_a)\geq \overline{\gamma_{\mu}}(\tilde a)$.
Hence, it  follows from \eqref{eqn:gamma-a-dayu-gamma-1} and Lemma \ref{lem:gamma-a-continuous-1}-$(iii)$ that whether
 $E(u_a)>\overline{\gamma_{\mu}}(\tilde a)$ or $E(u_a)=\overline{\gamma_{\mu}}(\tilde a)$, we can all obtain the desired contradiction. Thus, the claim holds.

 It remains to prove that $|(-\Delta)^{\frac{s_1}{2}}w_n|_{2}^{2}\to 0$ as $n\rightarrow\infty$.
 By above claim, we know that $|w_n|_{2}^{2} \to 0$.
Moreover,
 $\{w_n\}$ is bounded in $H^{s_1}(\mathbb{R}^{N})$. Hence, \eqref{eqn:G-N-equality} indicates that $|w_n|_{q}^{q} \to 0$. Moreover, from Lemma \ref{lem:vanishing-u-n-1}, we have
\begin{equation}\label{eqn:E-w-n-beta-0}
 E(w_n)\geq \alpha_0|(-\Delta)^{\frac{1}{2}}w_n |_2^2 +o_{n}(1).
\end{equation}
Recalling that
$E(u_n)=E(u_ a)+E(w_n)+o_{n}(1)= \overline{\gamma_{\mu}}(a)+o_n(1)$
and $u_a\in A_{a,\rho_0}$, we have $E(u_ a)\geq \overline{\gamma_{\mu}}(a)$  and then $E(w_n)\leq o_{n}(1)$.  From \eqref{eqn:E-w-n-beta-0}, we deduce that
$|(-\Delta)^{\frac{s_1}{2}}w_n |_2^2\to 0$ and the proof is complete.
\end{proof}

\noindent\textbf{Proof of Theorem \ref{Thm:q-mass-subcritical-p-sobolev-critical}.}
According to Lemma \ref{lem:hua-M-c-shouxlianxing}, we know that there exists $u_a\in S_a$ such that
$E(u_a)=\overline{\gamma_{\mu}}(a)<0$.
This means that problem \eqref{eqn:Mixed-fractioanl-CN} has a solution $u_a$ for some $\bar \lambda\in \mathbb R$, which is an interior local minimizer of the functional $E|_{S_a}$.
The properties of $u_a$ directly follow the symmetric decreasing rearrangement. Indeed,
let $|u_a |^{\ast}$ be the symmetric decreasing rearrangement of $u_a$,  then it clearly that
$|u_a|^*\in A_{a,\rho_0}$ and $E(|u_a|^*)\leq E(u_a)$. Thus, $u_a$ can be chosen to be a nonnegative and radially decreasing function.
Now we prove that it is a ground state. In view of the arguments in subsection \ref{sub:properties}, we know that if  $\mathcal P_0=\emptyset$, then there holds
$$
 \overline{\gamma_{\mu}}(a)=\inf_{\mathcal{P}}E=\inf_{\mathcal{P}_+}E,
$$
which implies that $u_a$ is a ground state because all critical points of $E|_{S_a}$ belong to the Pohozaev manifold. Hence, it suffices to verify $\mathcal P_0=\emptyset$.
Note that
$$
C_{N,s_1,p}=\mathcal S^{-\frac{2}{2_{s_1}^*}}\ \mbox{and}\ \kappa_p=s_1\ \mbox{for}\ p=2_{s_1}^*.
$$
By the proof of Lemma \ref{lem:P-natural-subcritical}, if we assume that
\begin{equation}\label{eqn:a-leq-bar-a-0}
\begin{aligned}
&a^{\frac{2qs_1 -N(q-2)}{4s_1 -N(q-2)}}\mu^{\frac{2 s_1}{4s_1 -N(q-2)}}\\
&\ \ <\bigg(\frac{s_1^2(2s_1^* -2)}{ C_{N,s_1,q} \kappa_q(s_1 2_{s_1}^* -q\kappa_q )}\bigg)^{\frac{2 s_1}{4s_1 -N(q-2)}} \bigg(\frac{\mathcal S^{\frac{2}{2_{s_1}^*}}(2s_1 -q\kappa_q)}{ 2_{s_1}^*s_1 - q\kappa_q }\bigg)^{\frac{N-2s_1}{4s_1}},
\end{aligned}
\end{equation}
then we can also obtain a contradiction and determine $\mathcal P_0=\emptyset$. Obviously, \eqref{eqn:a-leq-bar-a-0} is equivalent to
\begin{equation}\label{eqn:a-0-condition-2}\small
\begin{aligned}
a&<\mu^{-\frac{2 s_1}{2qs_1 -N(q-2)}}\bigg(\frac{s_1^2(2s_1^* -2)}{ C_{N,s_1,q} \kappa_q(s_1 2_{s_1}^* -q\kappa_q )}\bigg)^{\frac{2 s_1}{2qs_1 -N(q-2)}} \bigg(\frac{\mathcal S^{\frac{2}{2_{s_1}^*}}(2s_1 -q\kappa_q)}{ 2_{s_1}^*s_1 - q\kappa_q }\bigg)^{\frac{N-2s_1}{4s_1}\cdot{\frac{4s_1 -N(q-2)}{2qs_1 -N(q-2)}}}\\
&:=\bar a_0(\mu).
\end{aligned}
\end{equation}
Thus, if $a<\min\{a_0(\mu),\bar a_0(\mu)\}$, we conclude that $u_a$ is a ground state and the proof is complete.
\qed
\ \\
\ \\
\noindent\textbf{\Large Acknowledgements}

\noindent {This research is supported by National Natural Science Foundation of China (No.12371120).}
\vspace{0.8em}

\noindent\textbf{\Large Declaration}

\noindent {{\bf Conflict of interest} The authors do not have conflict of interest.}
\vspace{0.8em}

\noindent\textbf{\Large Data availability}

\noindent {No data was used for the research described in the article.}

\bibliographystyle{elsarticle-num}

\end{document}